\documentclass[10pt]{amsart}
\usepackage{amssymb,amsbsy,amsmath,amsfonts,times, amscd}
\usepackage{latexsym,euscript,exscale}
\usepackage{graphicx,color} \usepackage{amsthm}
\usepackage{fancyhdr} \usepackage{url}
\usepackage{epigraph}\usepackage{lmodern}
\usepackage{mathrsfs}\usepackage{cancel}
\usepackage[toc,page]{appendix}\usepackage[T1]{fontenc}
\usepackage[all,cmtip]{xy}\usepackage{mathtools}
\usepackage{enumerate}\usepackage{setspace}
\usepackage{helvet}
\usepackage{float}
\usepackage{mathabx} 
		
\numberwithin{equation}{section}
\newtheorem{theorem}{Theorem}
\newtheorem{thm}[theorem]{Theorem}
\newtheorem{cor}[theorem]{Corollary}
\newtheorem{lemma}[theorem]{Lemma}
\newtheorem{prop}[theorem]{Proposition}
\newtheorem*{maintheorem}{Theorem \ref{main}}
\newtheorem*{maincorollary}{Corollary \ref{corThmPrincipal}}
\theoremstyle{definition}
\newtheorem{defi}[theorem]{Definition}
\newtheorem{rem}[theorem]{Remark}
\newtheorem{problem}[theorem]{Problem}
\theoremstyle{remark}
\newtheorem{remark}[theorem]{Remark}

\numberwithin{theorem}{section}

\newcommand{\eps}{\varepsilon}

\newcommand{\Q}{\mathbb{Q}}
\newcommand{\N}{\mathbb{N}}
\newcommand{\R}{\mathbb{R}}

\newcommand{\tn}{{\vert\kern-0.25ex\vert\kern-0.25ex\vert}}
\newcommand{\Tn}{{\big\vert\kern-0.25ex\big\vert\kern-0.25ex\big\vert}}    
\newcommand{\TN}{{\Big\vert\kern-0.25ex\Big\vert\kern-0.25ex\Big\vert}}    

    \DeclareFontEncoding{FMS}{}{}
\DeclareFontSubstitution{FMS}{futm}{m}{n}
\DeclareFontEncoding{FMX}{}{}
\DeclareFontSubstitution{FMX}{futm}{m}{n}
\DeclareSymbolFont{fouriersymbols}{FMS}{futm}{m}{n}
\DeclareSymbolFont{fourierlargesymbols}{FMX}{futm}{m}{n}
\DeclareMathDelimiter{\VERT}{\mathord}{fouriersymbols}{152}{fourierlargesymbols}{147}

\begin{document}

\title[Coarse embeddings into superstable spaces]{Coarse embeddings into superstable spaces}
\author{B. M. Braga and A. Swift}
\address{Department of Mathematics, Statistics, and Computer Science (M/C 249)\\
University of Illinois at Chicago\\
851 S. Morgan St.\\
Chicago, IL 60607-7045\\
USA}\email{demendoncabraga@gmail.com}

\address{Department of Mathematics\\
Texas A\&M University\\
155 Ireland St.\\
College Station, TX 77843-3368\\
USA}\email{ats0@math.tamu.edu}

\date{}
\maketitle

\begin{abstract}
Krivine and Maurey proved in 1981 that every stable Banach space contains  almost isometric copies of $\ell_p$, for some $p\in[1,\infty)$. In 1983, Raynaud showed that if a Banach space uniformly embeds into a superstable Banach space, then $X$ must contain an isomorphic copy of $\ell_p$, for some $p\in[1,\infty)$. In these notes, we show that if a Banach space coarsely embeds into a superstable Banach space, then $X$ has a spreading model isomorphic to  $\ell_p$, for some $p\in[1,\infty)$. In particular, we obtain that there exist reflexive Banach spaces which do not coarsely embed into any superstable Banach space.  
\end{abstract}

\section{Introduction.}\label{SectionIntro}

 D. J. Aldous showed in \cite[Theorem 1.1]{Aldous1981}, that every infinite-dimensional subspace of $L_1$ contains an isomorphic copy of $\ell_p$, for some $p\in [1,\infty)$.  In order to generalize Aldous's result, J.  Krivine and B. Maurey introduced the notion of stable Banach space in \cite{KrivineMaurey1981}. To the authors' knowledge, stability was defined for general metric spaces by D. J. H. Garling \cite{Garling1981}.  A metric space $(M,d)$ is called \emph{stable} if
 
\[\lim_{i,\mathcal{U}}\lim_{j,\mathcal{V}}d(x_i,y_j)= \lim_{j,\mathcal{V}}\lim_{i,\mathcal{U}}d(x_i,y_j),\]\hfill
 
\noindent for all bounded sequences $(x_i)_{i=1}^\infty$ and $(y_j)_{j=1}^\infty$ in $M$, and all nonprincipal ultrafilters $\mathcal{U}$ and $\mathcal{V}$ over $\N$. A Banach space $X$ is called \emph{stable} if $(X,\|\cdot-\cdot\|)$ is stable as a metric space. As $L_p$ is stable for all $p\in [1,\infty)$ (see  \cite[Theorem II.2]{KrivineMaurey1981}), the following is a generalization of Aldous's result (see \cite[Theorem IV.1]{KrivineMaurey1981}).

\begin{thm}\textbf{(J. Krivine and B. Maurey, 1981)}
Let $X$ be a stable Banach space. There exists $p\in [1,\infty)$ such that $X$ contains a $(1+\eps)$-isomorphic copy of  $\ell_p$, for all $\epsilon>0$.
\end{thm}

In order to prove the theorem above, Krivine and Maurey looked  at \emph{types} on a stable Banach space $X$, i.e., functions $\sigma_a\colon X\to \R$ given by $\sigma_a(x)=\|x+a\|$, where $a$ is an element of some ultrapower of $X$. In \cite{KrivineMaurey1981}, the authors showed that every stable Banach space must contain what was called an \emph{$\ell_p$-type}, which results in the existence of almost isometric copies of $\ell_p$ inside $X$, for some $p\in [1,\infty)$.

As shown in \cite{Raynaud1983}, Krivine and Maurey's result can be extended to the nonlinear setting as follows. Let $(M,d)$ and $(N,\partial)$ be metric spaces, $f\colon M\to N$ be a map, and define

\begin{equation}\label{omega}
\omega_f(t)=\sup\{\partial(f(x),f(y))\mid d(x,y)\leq t\}
\end{equation}\hfill

\noindent and

\begin{equation}\label{rho}
\rho_f(t)=\inf\{\partial(f(x),f(y))\mid d(x,y)\geq t\}
\end{equation}\hfill

\noindent for all $t\geq 0$.

We say that $f$ is a \emph{uniform embedding} if $\lim_{t\to 0^+}\omega_f(t)=0$ (i.e., if $f$ is uniformly continuous) and $\rho_f(t)>0$, for all $t>0$ (i.e., if $f^{-1}$ exists and is uniformly continuous).  We say that $f$ is a \emph{uniform equivalence} if $f$ is a uniform embedding and $f$ is surjective.  We say that a Banach space $X$ is \emph{superstable} if every Banach space which is finitely representable in $X$ is also stable. Raynaud proved the following in \cite{Raynaud1983} (see the corollary on page 34 of \cite{Raynaud1983}).

\begin{thm}\textbf{(Y. Raynaud, 1983)}
If a Banach space $X$ uniformly embeds into  a superstable Banach space, then $X$ contains an isomorphic copy of $\ell_p$, for some $p\in [1,\infty)$. 
\end{thm}

 Raynaud's proof is also based on analyzing a space of types over the Banach space $X$. Precisely, the author shows that if $X$ uniformly embeds into a superstable Banach space, then there exists a translation-invariant stable metric $d$ on $X$ uniformly equivalent to the metric induced by the norm.  A metric $d$ defined for a Banach space $X$ is called \emph{translation-invariant} if $d(x,y)=d(x-y,0)$ for all $x,y\in X$.  Two metrics $d$ and $\partial$ defined for the same set $M$ are called \emph{uniformly equivalent} if the identity map $\mathrm{Id}\colon (M,d)\to (M,\partial)$ is a uniform equivalence.  Assuming $(X,\|\cdot\|)$ is separable, once one has a translation-invariant stable metric $d$ on $X$ uniformly equivalent to the norm, it is possible to define a space of types for $(X,d)$ similarly to the type space defined by Krivine and Maurey.
 Raynaud shows that the type space of $(X,d)$ must contain a so-called \emph{$\ell_p$-type}, for some $p\in[1,\infty)$, which implies the existence of an isomorphic copy of $\ell_p$ inside $(X,\|\cdot\|)$, for some $p\in [1,\infty)$. For more on stability and types on Banach spaces see  \cite{Garling1981}, \cite{G-D}, \cite{HaydonMaurey1986} and \cite{Iovino1998}.

The interest in Banach spaces as metric spaces and their nonlinear geometric properties has recently increased significantly, hence the question whether analogues of Raynaud's result hold for different kinds of nonlinear embeddings other than uniform embeddings becomes natural. Given metric spaces $(M,d)$ and $(N,\partial )$, and a map $f\colon M\to N$, we say that $f$ is \emph{expanding} if $\lim_{t\to\infty}\rho_f(t)=\infty$ and \emph{coarse} if $\omega_f(t)<\infty$, for all $t\geq 0$.  We say that $f$ is a \emph{coarse embedding} if $f$ is both coarse and expanding.  In this context,  N. Kalton asked the following in \cite[Problem 6.6]{Kalton2007}.

\begin{problem}\label{ProbKalton}
Assume that a Banach space $X$ coarsely embeds into a superstable Banach space. Does it follow that $X$ contains an isomorphic copy of $\ell_p$, for some $p\in [1,\infty)$?
\end{problem}

Although we were not able to obtain an answer to Kalton's problem, we obtain the following result in these notes.

\begin{theorem}\label{main}
If a Banach space $X$ coarsely embeds into a superstable Banach space, then $X$ has a basic sequence with an associated spreading model isomorphic to $\ell_p$, for some $p\in[1,\infty)$.
\end{theorem}

N. Kalton proved in \cite[Theorem 2.1]{Kalton2007}, that any stable metric space  embeds into some reflexive Banach space by a map which is both a uniform and a coarse embedding. In the same paper, Kalton   asked if the converse of this result also holds. Precisely,  the following is open (see \cite[Problem 6.1]{Kalton2007}).

\begin{problem}\label{ProblemReflexEmbStable}
Does every (separable) reflexive Banach space embed coarsely (resp. uniformly) into a stable space?
\end{problem}

By Raynaud's result, it is clear that there are separable reflexive spaces which do not uniformly embed into superstable spaces. However, to the best of our knowledge,  it was unknown whether every reflexive Banach space coarsely embeds into a superstable Banach space. As a corollary of Theorem \ref{main}, we obtain the following.

\begin{cor}\label{corThmPrincipal}
There are separable reflexive Banach spaces which do not coarsely embed into any superstable Banach space. 
\end{cor}

We now describe the organization of this paper. In Section \ref{SectionBackground}, we give all the background needed for these notes. Also in Subsection \ref{SubsectionBaire}, we prove Lemma \ref{bairelemma2}, which will be essential in the proof of our main theorem. In Section \ref{sectioninvariant},  we show that if $\varphi\colon X\to Y$ is a coarse map, then, by substituting $Y$ with an ultrapower of $\ell_1(Y)$, we may assume that $\|\varphi(x)-\varphi(y)\|=\|\varphi(x-y)\|$, for all $x,y\in X$ (see Theorem \ref{embTOultra}). This result allows us to obtain a translation-invariant (pseudo)metric $d$ on $X$ such that $\rho_\varphi (t)\leq \rho_\mathrm{Id}(t)$ and $\omega_\mathrm{Id}(t)\leq \omega_\varphi (t)$ for all $t\geq 0$, where $\mathrm{Id}\colon (X,\|\cdot\|)\to (X,d)$ is the identity map. This will play a fundamental role in our definition of the type space. Still in Section \ref{sectioninvariant}, we discuss the concepts of some weaker kinds of nonlinear embeddings (i.e., weakenings of uniform and coarse embeddings) and give a condition for the existence of a uniform embedding of the unit ball of a Banach space into a superstable Banach space (see Theorem \ref{ThmBall}). We finish this section with an open problem regarding these weaker notions of nonlinear embeddings (see Problem \ref{ProbEmbEq}).

In Section \ref{SectionType}, we define the type space which we will deal with in these notes, define the operations of dilation and convolution in this type space, and prove some basic properties of it. Section \ref{SectionConicClass}, Section \ref{SectionModel} and Section \ref{SectionCoarseApp} are quite technical and will provide us with all the tools needed to prove Theorem \ref{main}. Finally, in Section \ref{SectionCoarselp}, we conclude the proof of Theorem \ref{main}.

\section{Preliminaries.}\label{SectionBackground}
 
In these notes, we let $\N=\{n\}_{n=1}^\infty$ and $\N_0=\{0\}\cup \N$. The Banach space notation used in these notes is standard, and we refer the reader to \cite{AK} for more on that. For instance,  we denote the closed unit ball of a Banach space $X$ by $B_X$. Also, given $p\in [1,\infty)$ and $\overline{x}=(x_i)_{i=1}^N\in\R^{<\N}$, we define $\|\bar{x}\|_p=(\sum_{i=1}^N|x_i|^p)^{1/p}$ and $\|\bar{x}\|_\infty=\max\{|x_i|\mid 1\leq i\leq N\}$.

We define stability for metric spaces and superstability for Banach spaces as in Section \ref{SectionIntro}. By \cite[Theorem II.1]{KrivineMaurey1981} and \cite[Theorem 0.1]{Raynaud1983}, both stability and superstability are closed under taking $\ell_p$-sums, for $p\in[1,\infty)$. Precisely, given $p\in [1,\infty)$, if $X$ is stable (resp. superstable) Banach space, then $\ell_p(X)$ is also stable (resp. superstable). We will be using this property without mention in these notes.  In particular, $\ell_p$ is superstable for every $p\in[1,\infty)$.  Note however that $c_0$ is not even coarsely or uniformly embeddable into a stable metric space (see \cite{Kalton2007}).

We say that  $(M,d)$ is a \emph{pseudometric space} if $d\colon X\times X\to \R_+$ is a \emph{pseudometric}, i.e., if $d$ is a symmetric map satisfying the triangular inequality. We define stability for pseudometric spaces analogously to stability for metric spaces.  Given pseudometric spaces $(M,d)$, $(N,\partial)$ and a map $f\colon M\to N$, we define $\omega_f$ and $\rho_f$ by the formulas given in Equation \ref{omega} and Equation \ref{rho}, and define uniform and coarse embeddings analogously to the definitions in Section \ref{SectionIntro}. We say that $f$ is a \emph{coarse equivalence} if $f$ is a coarse embedding and $f$ is cobounded, i.e., $\sup_{y\in N}\partial(y,f(M))<\infty$.
Two pseudometrics $d$ and $\partial$ defined for the same set $M$ are called \emph{coarsely equivalent} if the identity map $\mathrm{Id}\colon (M,d)\to (M,\partial)$ is a coarse equivalence.

\subsection{Spreading models.}\label{SubsectionSpreading}
 
Let $X$ be a Banach space and $(x_n)_{n=1}^\infty$ be a bounded sequence without Cauchy subsequences, and let $\mathcal{U}$ be a nonprincipal  ultrafilter on $\N$. Then there exists a Banach space $(S,\tn\cdot\tn)$ containing $X$ and a sequence $(\zeta_n)_{n=1}^\infty$ in $S$ which is linearly independent of $X$ such that, for all $y\in X$, and all $\alpha_1,\ldots,\alpha_k\in \R$, we have 

\[\TN y+\sum_{j=1}^k\alpha_j\zeta_j \TN =\lim_{n_1,\mathcal{U}}\ldots\lim_{n_k,\mathcal{U}} \Big\|y+\sum_{j=1}^k\alpha_jx_{n_j} \Big\|.\]\hfill

\noindent Without loss of generality, $S=X\oplus \overline{\text{span}}\{\zeta_n\mid n\in\N\}$ (see \cite{G-D}, Chapter 2, Section 2, for a proof of this fact). The space $\overline{\text{span}}\{\zeta_n\mid n\in\N\}$ is called a \emph{spreading model of $(x_n)_{n=1}^\infty$} and the sequence $(\zeta_n)_{n=1}^\infty$ is called the \emph{fundamental sequence} of $\overline{\text{span}}\{\zeta_n\mid n\in\N\}$. If $X$ is separable, then by going to a subsequence of $(x_n)_{n=1}^\infty$ if necessary, we may assume that for every $\epsilon>0$, there is $N\in \mathbb{N}$ such that 

\[
\Bigg|\TN y+\sum_{j=1}^k\alpha_j\zeta_j \TN -\Big\|y+\sum_{j=1}^k\alpha_jx_{n_j} \Big\|\Bigg|<\epsilon
\]\hfill

\noindent  whenever $N\leq n_1<\dots <n_k$.  A fundamental sequence $(\zeta_n)_{n=1}^\infty$ of a spreading model is \emph{$1$-spreading}, i.e., $(\zeta_n)_{n=1}^\infty$ is $1$-equivalent to all of its subsequences. Also, the sequence $(\xi_n)_{n=1}^\infty$ is $1$-sign unconditional, where $\xi_n=\zeta_{2n-1}-\zeta_{2n}$, for all $n\in\N$ (see \cite[Proposition II.3.3]{G-D}). We refer to \cite{ArgyrosTodorcevic}, \cite{BeauzamyLepreste}, and \cite{G-D} for the theory of spreading models.

\subsection{Ultrapowers.}\label{SubsectionUltra} Let $X$ be a Banach space, $I$ be an index set, and $\mathcal{U}$  be a nonprincipal ultrafilter on  $I$. We define

\[X^I/\mathcal{U}=\Big\{(x_i)_{i\in I}\in X^I \mid \sup_{i\in I}\|x_i\|<\infty\Big\}/\sim ,\]\hfill

\noindent where $(x_i)_{i\in I}\sim (y_i)_{i\in I}$ if $\lim_{i, \mathcal{U}}\|x_i-y_i\|=0$. $X^I/\mathcal{U}$ is a Banach space with norm $\|x\|=\lim_{i,\mathcal{U}}\|x_i\|$, where $(x_i)_{i\in I}$ is a representative of the class $x\in X^I/\mathcal{U}$. By abuse of notation, we will not distinguish between $(x_i)_{i\in I}$ and its equivalence class. The space $X^I/\mathcal{U}$ is called an \emph{ultrapower} of $X$. 

Every ultrapower $X^I/\mathcal{U}$ of a Banach space $X$ is finitely representable in $X$ (see \cite[Proposition 11.1.12(i)]{AK}). On the other hand, if a separable Banach space $Y$ is finitely representable in $X$, then $Y$ is isometrically isomorphically embeddable into some ultrapower of $X$ (see \cite[Proposition 11.1.12(ii)]{AK}). Therefore, a Banach space $X$ is superstable if and only if all of its ultrapowers are stable.

\subsection{Baire class 1 functions.}\label{SubsectionBaire}

Let $X$ and $Y$ be metrizable topological spaces.
Recall that a subset of a topological space is called $F_\sigma$ if it is the countable union of closed sets, is called $G_\delta$ if it is the countable intersection of open sets, and is called \emph{comeager} if it is the countable intersection of sets with dense interiors.
A function $f\colon X\to Y$ is called \emph{Baire class 1} if the inverse image of any open subset of $Y$ under $f$ is an $F_\sigma$ subset of $X$.  If $Y$ is separable, and $f$ is Baire class 1, then the set of points of continuity for $f$ is a comeager $G_\delta$ subset of $X$.
If $Y$ is separable and $(f_n\colon X\to Y)_{n=1}^\infty$ is a sequence of Baire class 1 functions, then $(f_n)_{n=1}^\infty\colon X\to Y^\mathbb{N}$ is a Baire class 1 function.
The pointwise limit of a sequence of continuous functions from $X$ to $Y$ is a Baire class 1 function.
The restriction of a Baire class 1 function is a Baire class 1 function.
For proofs of these facts and more info about Baire class 1 functions, see \cite{Kechris1995} and \cite{Kuratowski1966}.

\begin{lemma}
\label{bairelemma}
Let $X$ be a metrizable $\sigma$-compact topological space, $Y$ a topological space, and let $f\colon X\times Y\to \mathbb{R}$ be separately continuous.  Given a metric $d$ inducing the topology of $X$ and a countable family $\mathcal{K}$ of compact subsets of $X$ such that $X=\bigcup_{K\in \mathcal{K}}K$;
if there is $\delta>0$ such that for each $x\in X$, $B_\delta (x)\cap K\neq \emptyset$ for only finitely many $K\in \mathcal{K}$, then $f$ is the pointwise limit of a sequence of continuous functions.   
\end{lemma}

\begin{proof}
For each $n\in \mathbb{N}$, let $\left\{x_{n,i}\right\}_{i=1}^\infty$ be a $\frac{\delta}{2(n+1)}$-dense set in $(X,d)$ such that $\left|\left\{x_{n,i}\right\}_{i=1}^\infty \cap  K\right|<\infty$ for every $K\in \mathcal{K}$.
For each $n,i\in \mathbb{N}$, define $g_{n,i}\colon X\to \mathbb{R}_+$ by $g_{n,i}(x)=\max \left\{\frac{\delta}{n+1}-d\left(x_{n,i},x\right),0\right\}$ for every $x\in X$.
Note that $g_{n,i}$ is continuous and given $x\in X$,  $g_{n,i}\restriction_{B_{\delta/2}(x)}$ is a nonzero function for some but only finitely many $i\in \mathbb{N}$.
Thus the function $h_{n,i}:=\frac{g_{n,i}}{\sum_{j=1}^\infty g_{n,j}}$ is well-defined and continuous.
For each $n\in \mathbb{N}$, define $f_n\colon X\times Y\to \mathbb{R}$ by 

\[f_n(x,y)=\sum_{i=1}^\infty f\left(x_{n,i},y\right)h_{n,i}(x)\]\hfill

\noindent for every $(x,y)\in X\times Y$ and note that $f_n$ is itself continuous by the separate continuity of $f$ and the observation on $g_{n,i}\restriction_{B_{\delta/2}(x)}$.
The sequence $(f_n)_{n=1}^\infty$ converges pointwise to $f$.
Indeed, take any $(x,y) \in X\times Y$ and any $\varepsilon>0$.
Let $N\in \mathbb{N}$ be such that $|f(x,y)-f(x',y)|<\varepsilon$ when $d(x, x')<\frac{\delta}{N}$.
Then, for $n\geq N$,

\begin{align*}
|f(x,y)-f_n(x,y)|&=\Big|\sum_{i=1}^\infty\left(f(x,y)-f\left(x_{n,i},y\right)\right)h_{n,i}(x)\Big|\\
&\leq \sum_{i=1}^\infty\left|f(x,y)-f\left(x_{n,i},y\right)\right|h_{n,i}(x)\\
&<\varepsilon \cdot\sum_{i=1}^\infty h_{n,i}(x)\\
&=\varepsilon.
\end{align*}
\end{proof}

Given a set $X$ and a family $\mathcal{F}$ of functions from $X\times X$ to $X$, define the sequence $\left(\mathcal{F}^{[k]}\right)_{k=1}^\infty$ of subsets of $X^X$ recursively by

\begin{align*}
\mathcal{F}^{[0]}&=\{x\mapsto x\}\\
\mathcal{F}^{[k+1]}&=\big\{x\mapsto f(x,g(x))\ |\ f\in \mathcal{F}, g\in \mathcal{F}^{[k]}\big\}.
\end{align*}\hfill

The following lemma will give us Lemma \ref{commonpoint} below, which is essential for the proof of Theorem \ref{main}.

\begin{lemma}
\label{bairelemma2}
Let $X$ be a separable metric space and $\mathcal{F}$ a countable family of Baire class 1 functions from $X\times X$ to $X$.  There is a comeager $G_{\delta}$ subset $E$ of $X$ such that $g$ is continuous on $E$ for all $g\in \bigcup_{k=1}^\infty \mathcal{F}^{[k]}$.
\end{lemma}

\begin{proof}
Certainly, $g$ is continuous on $E_0=X$ for  $g\in \mathcal{F}^{[0]}$.
Suppose $k\in \mathbb{N}_0$ is such that there is a comeager $G_\delta$ subset $E_k$ of $X$ such that $g$ is continuous on $E_k$ for all $g\in \mathcal{F}^{[k]}$.
For each $g\in \mathcal{F}^{[k]}$, let $\Gamma_g=\{(x,g(x))\ |\ x\in E_k\}$.
Since $\mathcal{F}$ is a countable family of Baire class 1 functions with separable codomain $X$, there is a comeager $G_\delta$ subset $F_g$ of $\Gamma_g$ such that $f\restriction_{\Gamma_g}$ is continuous on $F_g$ for all $f\in \mathcal{F}$.
Let $\pi\colon X\times X\to X$ be the first coordinate projection.
Consider $U=\Gamma_g\cap V\times W$, where $V, W$ are open subsets of $X$; and suppose $x\in \pi(U)$, so that $(x,g(x))\in U$.
As $W$ is open and $g(x)\in W$, there is $r_1>0$ such that $B_{r_1}(g(x))\subseteq W$.
Since $g$ is continuous on $E_k$, there is $r_2>0$ such that $g(B_{r_2}(x)\cap E_k)\subseteq B_{r_1}(g(x))$.
Thus $(V\cap B_{r_2}(x))\cap E_k$ is an open neighborhood of $x$ in $E_k$ contained in $\pi(U)$.
Since $x\in \pi(U)$ was arbitrary, $\pi(U)$ is open in $E_k$.
And $U$ was an arbitrary element in a basis for the topology on $\Gamma_g$, so $\pi(U)$ is open in $E_k$ whenever $U$ is open in $\Gamma_g$.
It follows easily that $\pi(F_g)$ is a comeager $G_\delta$ subset of $E_k$ since $F_g$ is a comeager $G_\delta$ subset of $\Gamma_g$.
Let $E_{k+1}=\bigcap_{g\in \mathcal{F}^{[k]}}\pi(F_g)$.
Since $\mathcal{F}^{[k]}$ is countable, $E_{k+1}$ is a comeager $G_\delta$ subset of $E_k$, and therefore also of $X$, since $E_k$ is a comeager $G_\delta$ subset of $X$.
Now take any $g\in \mathcal{F}^{[k+1]}$.
Then there is $f\in \mathcal{F}$ and $g'\in \mathcal{F}^{[k]}$ such that $g(x)=f(x,g'(x))$ for all $x\in X$.
And if $x\in E_{k+1}$, then by construction $x$ is a point of continuity for $g'$ and $(x,g'(x))$ is a point of continuity for $f\restriction_{\Gamma_{g'}}$.
Therefore $x$ is a point of continuity for $g$.
Thus, we have constructed a comeager $G_\delta$ subset $E_{k+1}$ of $E_k$ such that $g$ is continuous on $E_{k+1}$ for all $g\in \mathcal{F}^{[k+1]}$. 
And so we can recursively define such $E_k$ for all $k\in \mathbb{N}$.
The result follows by taking $E=\bigcap_{k=0}^\infty E_k$.
\end{proof}

\section{Making coarse maps ``invariant$"$.}\label{sectioninvariant}

In this section, we use Markov-Kakutani's fixed-point theorem to show that coarse embeddings may be modified and made more ``tamed$"$ if we allow ourselves to substitute its codomain by an ultrapower of the $\ell_1$-sum of the original space. Precisely, we have the following.

\begin{thm}\label{embTOultra}
Let $X$ and $Y$ be Banach spaces and  $f\colon X\to Y$ a coarse map. Then there exists a nonprincipal ultrafilter $\mathcal{U}$ on an index set $I$, and a map $F\colon X\to \ell_1(Y)^I/\mathcal{U}$, such that

\[\rho_f(\|x-y\|)\leq \|F(x)-F(y)\|\leq \omega_f(\|x-y\|),\]\hfill

\noindent and 

\begin{equation*}
\|F(x)-F(y)\|=\|F(x-y)\|,\ \ \text{ for all }\ \ x,y\in X.
\end{equation*}
\end{thm}

\begin{proof}
 Let 

\[C= \prod_{(x,y)\in X\times X} [\rho_f(\|x-y\|),\omega_f(\|x-y\|)].\]\hfill

\noindent We consider $C$ as a subset of $\mathbb{R}^{X\times X}$ with its pointwise convergence topology.  By the assumption that $f$ is coarse and Tychonoff's theorem, $C$ is compact.  We denote elements of $C$ by $D(x,y)$.  Let $d\colon X\times X\to \R$ be given by $d(x,y)= \|f(x)-f(y)\|$, for all $x,y\in X$. So, $d\in C$.

For each $z\in X$, define $\hat{z}\colon \mathbb{R}^{X\times X}\to \mathbb{R}^{X\times X}$ by letting $\hat{z}(g)(x,y)=g(x+z,y+z)$ for all $g\in \mathbb{R}^{X\times X}$ and all $x,y\in X$. Let $A=\overline{\mathrm{conv}}\{\hat{z}(d)\mid z\in X\}\subseteq \mathbb{R}^{X\times X}$. By the definition of the pointwise convergence topology on $\R^{X\times X}$, we have that $A\subseteq C$. The family $\{\hat{z}\restriction_A\}_{z\in X}$ is easily seen to be a commuting family of continuous, affine self-mappings of the compact convex subset $A$ of $\mathbb{R}^{X\times X}$.
Hence, by Markov-Kakutani's fixed-point theorem,  there exists $D\in A$ such that $\hat{z}(D)=D$ for all $z\in X$.
That is, $D(x+z,y+z)=D(x,y)$ for all $x,y,z\in X$.
Say $D=\lim_{i\in\mathcal{U}} D_i$, where $I$ is an index set,  $\mathcal{U}$ is some nonprincipal ultrafilter on $I$, and   $D_i\in \mathrm{conv}\{\hat{z}(d)\mid z\in X\}$, for all $i\in I$. For each $i\in I$, we have that  $D_i=\sum_{j=1}^{s(i)} \alpha_{i,j}\hat{z}_{i,j}(d)$, for some finite sequence  $(\alpha_{i,j})_{j=1}^{s(i)}$ of nonnegative real numbers such that $\sum_{j=1}^{s(i)}\alpha_{i,j}=1$, and some finite sequence $(z_{i,j})_{j=1}^{s(i)}$  in $X$.

For $y_1,\ldots, y_n\in Y$, we denote $(y_1,\ldots,y_n,0,0,\ldots)\in \ell_1(Y)$ by $\oplus_{j=1}^n y_j$. Consider the map

\begin{align*}
F=(F_i)_{i\in I}\colon  X&\longrightarrow \ell_1(Y)^I/\mathcal{U}\\
x&\longmapsto \Big( \bigoplus_{j=1}^{s(i)}\alpha_{i,j}\big(f(x+z_{i,j})-f(z_{i,j})\big)\Big)_{i\in I}.
\end{align*}\hfill

\noindent As $\sup_{i\in I}\|F_i(x)\|_{\ell_1(Y)}\leq \omega_f(\|x\|)$, for all $x\in X$, the map $F$ is well-defined.  By the definition of the norm on $\ell_1(Y)^I/\mathcal{U}$, we have that

\[\|F(x)-F(y)\|_{\ell_1(Y)^I/\mathcal{U}}=D(x,y),\]\hfill

\noindent  for all $x,y\in X$. Therefore, as $D(x,y)=D(x-y,0)$, for all $x,y\in X$, and $F(0)=0$, we are done. 
\end{proof}

\begin{cor}
\label{invmetric2}
Let $(X,\|\cdot\|)$ be a Banach space, and $Y$ be a superstable Banach space.  Let  $f\colon X\to Y$ be a coarse map. Then there exists a translation-invariant stable pseudometric $d$ on $X$ such that 

\[\rho_f(\|x-y\|)\leq d(x,y)\leq \omega_f(\|x-y\|), \ \text{ for all }\  x,y\in X.\]\hfill

\noindent In particular, if $X$ coarsely embeds into a superstable space, then  there exists a translation-invariant stable pseudometric $d$ on $X$ such that the identity map $\mathrm{Id}\colon (X,\|\cdot\|)\to (X,d)$ is a coarse equivalence.
\end{cor}

\begin{proof}
Let  $F\colon X\to \ell_1(Y)^I/\mathcal{U}$ be obtained from  Theorem \ref{embTOultra} applied to $f$. Define a map $d\colon X\times X\to \mathbb{R}_+$  by setting $d(x,y)=\|F(x)-F(y)\|$ for all $x,y\in X$.  It can easily be seen that $d$ is a translation-invariant pseudometric on $X$ and that $\rho_f(\|x-y\|)\leq d(x,y)\leq \omega_f(\|x-y\|)$ for all $x,y\in X$. Also, as $F$ is coarse and $\ell_1(Y)^I/\mathcal{U}$ is stable, it follows that  $d$ is a stable metric.
\end{proof}

The corollary above is the analogous version of \cite[Theorem 0.2]{Raynaud1983} to our setting. In \cite[Theorem 0.2]{Raynaud1983}, Raynaud proved that if a Banach space $X$ uniformly embeds into a superstable Banach space $Y$, then $X$ has a translation-invariant stable metric which is uniformly equivalent to $X$'s norm. However, Raynaud's proof relies on an averaging process which relies on the emebedding  $f\colon X\to Y$  being uniformly continuous. By using Markov-Kakutani's fixed-point theorem, we are able to overcome that. Clearly, Corollary \ref{invmetric2} implies \cite[Theorem 0.2]{Raynaud1983}.

\begin{rem}\label{metricmetric}
Although this will not be needed for the main result in these notes, Corollary \ref{invmetric2} can actually be improved to show the existence of a coarsely equivalent translation-invariant stable \emph{metric} on $X$. Indeed, it has been shown by the first named author (see \cite[Theorem 1.6]{BragaEmb}) that if $X$ and $Y$ are Banach spaces and $f\colon X\to Y$ is a coarse embedding, then there is a coarse embedding $\hat{f}\colon X\to \ell_1(Y)$ with uniformly continuous inverse (meaning $\rho_{\hat{f}}(t)>0$ whenever $t>0$).
Thus, the same proof as in Corollary \ref{invmetric2} with $\ell_1(Y)$ replacing $Y$ and $\hat{f}$ replacing $f$ will yield that $\text{Id}\colon (X,\|\cdot\|)\to (X,d)$ is a coarse embedding with uniformly continuous inverse. In particular, $d$ is a metric.
\end{rem}

Although Kalton proved that there exist separable Banach spaces which are coarsely equivalent but not uniformly equivalent (see \cite[Theorem 8.9]{Kalton2012}), the same problem remains open for embeddings. Precisely, given Banach spaces $X$ and $Y$, does $X$ coarsely embed into  $Y$ if and only if $X$ uniformly embeds into $Y$?  It is not even known whether the coarse embeddability of $X$ into $Y$ would be strong enough to give us that $B_X$ uniformly embeds into $Y$. In the remainder of this section, we use Theorem \ref{embTOultra} to prove a result on the uniform embeddability of the ball of a given Banach space into a superstable space (Theorem \ref{ThmBall}).

 Let $X$ and $Y$ be Banach spaces and consider a map $f\colon X\to Y$. For each $t\geq 0$, we define the \emph{exact compression modulus of $f$} as

\[\overline{\rho}_f(t)=\inf\{\|f(x)-f(y)\|\mid \|x-y\|=t\}.\]\hfill

\noindent The map $f$ is called \emph{almost uncollapsed} if there exists some $t>0$ such that $\overline{\rho}_f(t)>0$. This is equivalent to

\[\sup_{t>0}\inf_{\|x-y\|=t}\|f(x)-f(y)\|>0.\]\hfill

\noindent We say that $f\colon X\to Y$ is \emph{solvent} if for each $n\in\N$, there exists $R>0$, such that 

\[\|x-y\|\in [R,R+n] \ \ \text{ implies }\ \  \|f(x)-f(y)\|>n,\]\hfill

\noindent for all $x,y\in X$. By \cite[Lemma 60]{Rosendal2016}, a coarse map $f\colon X\to Y$ is solvent if and only if $\sup_{t>0}\overline{\rho}_f(t)=\infty$. It is clear that the condition of a map $f\colon X\to Y$ being almost uncollapsed is a weakening of $f$ having a uniformly continuous inverse, and the condition of $f$ being solvent is a weakening of $f$ being expanding (see Problem \ref{ProbEmbEq} below for more on that). We refer to \cite{Rosendal2016} and \cite{BragaWeak} for more on almost uncollapsed and solvent maps.

\begin{remark}\label{Reminvmetric2}
Theorem \ref{invmetric2} remains valid with $\rho_f$ replaced by $\overline{\rho}_f$. Indeed, the exact same proof will work. 
\end{remark}

\begin{thm}\label{ThmBall}
If a Banach space $X$ maps into a superstable space by a  map  which is both uniformly continuous  and almost uncollapsed, then $B_X$ uniformly  embeds into a superstable space.
\end{thm}

Before proving Theorem \ref{ThmBall}, we need the following proposition.

\begin{prop}\label{propinjctive}
Let $X$ and $Y$ be Banach spaces and $f\colon X\to Y$ be a solvent map such that $\|f(x)-f(y)\|=\|f(x-y)\|$ for all $x,y\in X$. Then, for every norm-bounded subset $B\subseteq X$, $f\restriction_B$ has a Lipschitz inverse. 
\end{prop}

\begin{proof}
First notice that, 

\[\|f(x)\|=\|f(x)-f(0)\|=\|f(0)-f(x)\|=\|f(-x)\|,\]\hfill

\noindent for all $x\in X$. Therefore, 

\[\|f(mx)\|=\|f((m-1)x)-f(-x)\|\leq \|f((m-1)x)\|+\|f(x)\|,\]\hfill

\noindent for all $x\in X$, and all $m\in\N$. So, $\|f(mx)\|\leq m\cdot\|f(x)\|$, for all $x\in X$, and all $m\in\N$. Let $N\in\N$ be such that $ B\subseteq N\cdot B_X$. As $f$ is solvent, we can find $n,R>2N$ such that $\|x\|\in[R,R+n]$ implies $\|f(x)\|>n$ (indeed, choose $R'$ as in the definition for $4N+2$, and then let $n=2N+1$ and $R=R'+2N+1$). By our choice of $n$ and $R$, for each $x\in 2N\cdot B_X$ we can pick $m_x\in\N$ such that $\|m_xx\|\in [R,R+n]$.  Hence,

\[\|f(x)\|>\frac{n}{m_x}\geq \frac{n}{R+n}\|x\|,\]\hfill

\noindent for all $x\in 2 N\cdot B_X$. This gives us that $\|f(x)-f(y)\|\geq \frac{n}{R+n}\|x-y\|$, for all $x,y\in B$. 
\end{proof}

\begin{proof}[Proof of Theorem \ref{ThmBall}.]
If $X$ maps into a superstable space by a uniformly continuous almost uncollapsed map, then, by \cite[Proposition 63]{Rosendal2016}, $X$ maps into a superstable space  by a map which is both  uniformly continuous and solvent. Hence, by Remark \ref{Reminvmetric2}, $X$ maps into a superstable space $Y$ by a uniformly continuous solvent map $F$ such that $\|F(x)-F(y)\|=\|F(x-y)\|$, for all $x,y\in X$. By Proposition \ref{propinjctive}, $F\restriction_{B_X}$ has a Lipschitz inverse. In particular, $B_X$ uniformly embeds into a superstable space.
\end{proof}

As we mentioned before,  it remains open whether coarse and uniform embeddability are equivalent in the context of Banach spaces. However, much more remains unknown regarding embeddability between Banach spaces. Precisely, the following is open (see \cite{BragaWeak} for more on that).

\begin{problem}\label{ProbEmbEq}
Let $X$ and $Y$ be Banach spaces. Are the following equivalent?

\begin{enumerate}[(i)]
\item $X$ uniformly embeds into $Y$.
\item $X$ coarsely embeds into $Y$.
\item $X$ maps into $Y$ by a map which is both a uniform and a coarse embedding.
\item $X$ maps into $Y$ by a map which is uniformly continuous and almost uncollapsed.
\item $X$ maps into $Y$ by a map which is coarse and solvent.
\end{enumerate}
\end{problem}

\section{Type space.}\label{SectionType}

Following Raynaud, our strategy for proving Theorem \ref{main} is to first make an appropriate definition for the type space of a Banach space coarsely embeddable into a superstable Banach space.  This type space needs to have certain compactness properties and needs to be able to not only reflect the metric structure of the Banach space, but also the algebraic structure.  Using compactness and methods commonly employed in the proof of Krivine's theorem, we'll be able to show the existence of a type that satisfies a nice $\ell_p$-inequality, and then push this back down onto the Banach space.

To motivate the definition of the type space, first consider a general metric space $(M,d)$.  One may ask whether $M$ can be compactified in a way that preserves the metric structure on $M$.  That is, under what conditions will there exist a compact metrizable space $\mathcal{T}$ such that $M$ homeomorphically maps onto a dense subset of $\mathcal{T}$?  Separability is certainly a necessary condition, and given that $\mathrm{Lip}_1(M)$ (the space of all real-valued Lipschitz functions over $M$ with Lipschitz constant less than or equal to $1$) is metrizable and closed in $\mathbb{R}^M$ under the pointwise-convergence topology when $M$ is separable, a natural $\sigma$-locally compact metrizable $\mathcal{T}$ that contains a dense homeomorphic copy of $M$ is the closure of $\{\overline{x}\}_{x\in M}$ in $\mathbb{R}^M$, where $\overline{x}$ is defined for all $x\in M$ by $\overline{x}(y)=d(x,y)$ for all $y\in M$.
If $d$ is a bounded metric, then $\mathcal{T}$ is in fact compact, and since every metric is topologically equivalent to a bounded metric, separability is also a sufficient condition.

Supposing now that $M$ is a vector space, and $\lim_{n\to \infty}\overline{x}_n$, $\lim_{n\to \infty} \overline{y}_n$ both exist, one may further ask under what conditions do $\lim_{n\to \infty} \overline{(x_n+y_n)}$ and $\lim_{n\to \infty} \overline{(\alpha x_n)}$ exist, where $\alpha$ is some scalar.  
Stability of $d$ is enough to show the existence of $\lim_{m\to \infty}\lim_{n\to \infty} \overline{x}_n(z-y_m)$ for any $z\in M$, and if $d$ is also translation-invariant, this means $\lim_{n\to \infty}\overline{(x_n+y_n)}$ exists after taking an appropriate subsequence.  
If $d$ is induced by a norm then $\lim_{n\to \infty} \overline{(\alpha x_n)}$ certainly exists since $\overline{(\alpha x_m)}(y)=|\alpha| \overline{x}_n(y/\alpha)$.  Otherwise, a slight modification needs to be made to $\mathcal{T}$.  
One must now account for scalars by defining $\mathcal{T}$ to be a subset of $\mathbb{R}^{\mathbb{F}\times M}$, where $\mathbb{F}$ is the field of scalars and $\overline{x}(\lambda, y)=d(\lambda x, y)$ for all $(\lambda, y)\in \mathbb{F}\times M$.  
Now, in this setting, $\lim_{n\to \infty} \overline{(\alpha x_n)}$ exists since $\overline{(\alpha x_n)}(\lambda, y)=\overline{x}_n(\lambda \alpha, y)$.  With these ideas in mind, we are now ready to explicitly define the type space we need.  For a more complete discussion of some of the ideas above, see \cite{Garling1981}.

From now on, we consider a separable infinite-dimensional Banach space $(X,\|\cdot\|)$ which admits a translation-invariant stable pseudometric $d$ coarsely equivalent to the metric induced by $\|\cdot\|$, and the corresponding identity map $\mathrm{Id}\colon (X,\|\cdot\|)\to (X,d)$.  By Corollary \ref{invmetric2}, such $d$ exists as long as $X$ coarsely embeds into a superstable space. 

\begin{remark}
Notice that, by Remark \ref{metricmetric}, we can actually assume that $d$ is a metric. However, in order to obtain the isomorphism constant in Remark \ref{constant} below, we need to work with  $d$ being the pseudometric given by Corollary \ref{invmetric2}. 
\end{remark}

\begin{remark}
Our definition of the type space will be similar to Raynaud's, with a few changes to the proofs resulting from having a metric that is coarsely equivalent rather than uniformly equivalent to the metric induced by the norm on $X$.  Note in particular that, in our case, a sequence may be dense in $(X,\|\cdot\|)$ while not being dense in $(X,d)$.  Thus, in order to have metrizability, we must use a countable subset of $X$ to define the type space.
\end{remark}

Let $\Delta$ be a countable $\|\cdot\|$-dense $\Q$-vector subspace of $X$. 
Given $x\in \Delta$, define the function $\overline{x}\in \mathbb{R}_+^{\mathbb{Q}\times \Delta}$ by $\overline{x}(\lambda, y)=d(\lambda x,y)$ for all $(\lambda, y)\in \mathbb{Q}\times \Delta$.  
The \emph{space of types} on $(\Delta, d\restriction_{\Delta\times \Delta})$, which we  denote by $\mathcal{T}$, is defined to be the closure of $\{\overline{x}\}_{x\in \Delta}$ in $\mathbb{R}^{\mathbb{Q}\times \Delta}$ (with the topology of pointwise convergence).
An element $\sigma$ of $\mathcal{T}$ is called a \emph{type}, and is called a \emph{realized type} if $\sigma=\overline{x}$ for some $x\in \Delta$, in which case $\sigma$ is also called the \emph{type realized by x}.  The type $\overline{0}$ is called the \emph{null} or \emph{trivial} type.

Note that the countability of $\mathbb{Q}\times \Delta$ implies that $\mathcal{T}$ is metrizable, and so every $\sigma\in \mathcal{T}$ can be expressed as $\lim_{n\to \infty}\overline{x}_n$ for some sequence $(x_n)_{n=1}^\infty$ in $\Delta$.  
Such a sequence is called a \emph{defining sequence} for $\sigma$.
Note also that in this case $\sigma(\lambda,x)=\lim_{n,\mathcal{U}}d(\lambda x_n,x)$ for every $(\lambda,x)\in \mathbb{Q}\times \Delta$ and every nonprincipal ultrafilter $\mathcal{U}$ over $\mathbb{N}$.
In particular, $\lim_{n\to \infty}d(x_n,0)$ exists, and so $(x_n)_{n=1}^\infty$ is a $d$-bounded (and therefore also $\|\cdot\|$-bounded) sequence in $\Delta$.

For every $ M\in \R_+$,  we let $\mathcal{T}_{M}=\{\sigma\in \mathcal{T}\mid \sigma(1,0)\leq M\}$. We will need the following lemma.

\begin{lemma}\label{T_Mcompact}
Say $M\in \mathbb{R}_+$. Then $\mathcal{T}_{M}$ is compact.
\end{lemma}

\begin{proof}
Say $\sigma \in \mathcal{T}_M$, and $(x_n)_{n=1}^\infty$ is a defining sequence for $\sigma$. 
As $\lim_{n\to \infty}d(x_n,0)=\sigma(1,0)\leq M$, we may suppose that the defining sequence for $\sigma$ is contained in the $d$-ball of radius $M+1$ around $0$. As $\text{Id}\colon (X,\|\cdot\|)\to (X,d)$ is expanding, there exists 
 $R<\infty$  such that $t\leq R$ whenever $\rho_{\mathrm{Id}}(t)\leq M+1$. Then, since $\rho_{\mathrm{Id}}(\|x_n\|)\leq d(x_n,0)\leq M+1$ for every $n\in \mathbb{N}$, we have 

\[\sigma(\lambda,x)=\lim_{n}d(\lambda x_n,x)\leq \lim_{n} (d(\lambda x_n,0)+d(0,x))\leq \omega_{\mathrm{Id}}(|\lambda| R)+d(0,x)\]\hfill

\noindent for all $(\lambda,x)\in \mathbb{Q}\times \Delta$.
That is, we have

\[\mathcal{T}_M\subseteq  \prod_{(\lambda,x)\in \mathbb{Q}\times\Delta}[0,\omega(|\lambda| R)+d(x,0)],\]\hfill

\noindent since $\sigma\in \mathcal{T}_M$ was arbitrary.
By Tychonoff's theorem and the fact that $\mathcal{T}_M$ is closed, we are finished.
\end{proof}

\begin{cor}
\label{TM_locallycompact}
The metric space $\mathcal{T}$ is $\sigma$-locally compact.
\end{cor}

The next lemma will allow us to define analogues of scalar multiplication and vector addition in the type space, capturing some of the algebraic structure of $X$.

\begin{lemma}
\label{welldef}
Suppose $\sigma,\tau\in \mathcal{T}$.
Then if $(w_n)_{n=1}^\infty, (x_n)_{n=1}^\infty$ are defining sequences for $\sigma$ and  $(y_n)_{n=1}^\infty, (z_n)_{n=1}^\infty$ are defining sequences for $\tau$, then

\begin{enumerate}[(i)]
\item The limits $\lim_{n}\overline{(\alpha w_n)}$ and $\lim_{n}\overline{(\alpha x_n)}$ exist and are equal for every $\alpha\in \mathbb{Q}$.\\
\item The limits $\lim_{n}\lim_{m}\overline{(w_n+y_m)}$ and $\lim_{n}\lim_{m}\overline{(x_n+z_m)}$ exist and are equal.\\
\end{enumerate}
\end{lemma}

\begin{proof} Item (i) follows immediately from the definitions.  By a straightforward argument using the translation-invariance and stability of $d$, item (ii) also follows.
\end{proof}

\begin{defi}
Let $\sigma, \tau\in \mathcal{T}$ and let $(x_n)_{n=1}^\infty, (y_m)_{m=1}^\infty$ be any defining sequences for $\sigma$ and $ \tau$, respectively. We define the  \emph{dilation} operation on $\mathcal{T}$ by $(\alpha,\sigma)\in \Q\times\mathcal{T}\mapsto\alpha\cdot\sigma\in \mathcal{T}$, where   $\alpha \cdot \sigma\coloneqq \lim_{n}\overline{(\alpha x_n)}$.
We define the \emph{convolution} operation on $\mathcal{T}$ by $(\sigma,\tau)\in\mathcal{T}\times\mathcal{T}\mapsto\sigma*\tau\in\mathcal{T}$, where $\sigma \ast \tau\coloneqq \lim_{n}\lim_{m}\overline{(x_n+y_m)}$.
By Lemma \ref{welldef}, both dilation and convolution are well-defined.  For $(\sigma_j)_{j=1}^k\subseteq \mathcal{T}$, we define $\bigast_{j=1}^k\sigma_j$ in the obvious way, and we allow dilation to bind more strongly than convolution in our notation, i.e., we write $\alpha \cdot\sigma \ast \tau$ to mean $(\alpha \cdot \sigma)\ast \tau$.
\end{defi}

It follows easily from the definitions that, given $\sigma \in \mathcal{T}$ and a defining sequence $(x_n)_{n=1}^\infty$ for $\sigma$, we have $\alpha \cdot \sigma (\lambda,x)=\sigma(\lambda \alpha, x)$ for every $(\lambda,x)\in \mathbb{Q}\times \Delta$ and $\sigma\ast\tau=\lim_{n\to \infty}\overline{x}_n\ast\tau$ for every $\tau\in \mathcal{T}$. Furthermore, using the translation-invariance and stability of $d$, it is easily shown that convolution is associative and commutative, and that dilation distributes over convolution.  We now prove some continuity properties of our dilation and convolution maps.

\begin{lemma}
\label{rightcont}
Dilation is a right-continuous map from $\mathbb{Q}\times\mathcal{T}$ to $\mathcal{T}$.
\end{lemma}

\noindent
\begin{proof}
Fix $\alpha\in \mathbb{Q}$ and suppose $(\sigma_n)_{n=1}^\infty$ is a sequence in $\mathcal{T}$ converging to $\sigma\in \mathcal{T}$.
Then $\alpha \cdot \sigma (\lambda, x) =\sigma (\lambda \alpha ,x)=\lim_{n\to \infty}\sigma_n(\lambda \alpha,x)=\lim_{n\to \infty}\alpha\cdot \sigma_n(\lambda,x)$
for all $(\lambda, x)\in \mathbb{Q}\times \Delta$.
Thus $\alpha \cdot\sigma =\lim_{n\to \infty}\alpha \cdot\sigma_n$. This was for an arbitrary converging sequence in $\mathcal{T}$, so dilation is right continuous.
\end{proof}

\begin{lemma}
\label{sepcont}
Convolution is a separately continuous map from $\mathcal{T}\times \mathcal{T}$ to $\mathcal{T}$.
\end{lemma}

\noindent
\begin{proof}
Let $D$ be a metric compatible with the topology on $\mathcal{T}$.  Fix $\tau \in \mathcal{T}$ and suppose $(\sigma_n)_{n=1}^\infty$ is a sequence in $\mathcal{T}$ converging to $\sigma\in \mathcal{T}$. For each $n\in \mathbb{N}$, let $(x_{n,m})_{m=1}^\infty$ be a defining sequence for $\sigma_n$, and let $m_n\in \mathbb{N}$ be such that $D(\sigma_n, \overline{x}_{n,m_n})<\frac{1}{n}$ and $D(\overline{x}_{n,m_n}\ast\tau,\sigma_n\ast\tau)<\frac{1}{n}$.
Then $(x_{n,m_n})_{n=1}^\infty$ is a defining sequence for $\sigma$ by the triangle inequality; and so, again by triangle inequality, $\sigma\ast\tau
=\lim_{n}\sigma_n\ast\tau$.
This was for an arbitrary converging sequence in $\mathcal{T}$, so convolution (which is commutative) is separately continuous.
\end{proof}

\begin{cor}
\label{conv_is_baire}
Convolution is a Baire class 1 map from $\mathcal{T}\times\mathcal{T}$ to $\mathcal{T}$.
\end{cor}

\begin{proof}
Given $(\lambda,x)\in \mathbb{Q}\times \Delta$, let $\Phi_{\lambda,x}\colon \mathcal{T}\times\mathcal{T}\to \mathbb{R}$ be defined by $\Phi_{\lambda,x}(\sigma,\tau)=\sigma\ast\tau(\lambda,x)$ for all $\sigma,\tau\in \mathcal{T}$.
Choose a compatible metric $D$ for the topology on $\mathcal{T}$ and note that there is $\delta>0$ such that $D(\sigma, \tau)\geq \delta$ whenever $|\sigma(1,0)-\tau(1,0)|$ is large enough.
Now, by Lemma \ref{sepcont} and the topology on $\mathcal{T}$, $\Phi_{\lambda,x}$ is separately continuous; and by Lemma \ref{T_Mcompact}, $\mathcal{T}_M$ is compact for every $M\in \mathbb{R}_+$.
Thus; applying Lemma \ref{bairelemma} with $X=Y=\mathcal{T}$, $f=\Phi_{\lambda,x}$, $d=D$, $\mathcal{K}=\{\mathcal{T}_{M+1}\setminus \mathrm{int}\mathcal{T}_M\}_{M=0}^\infty$, and with $\delta$  as above in the statement of Lemma \ref{bairelemma}; we have that $\Phi_{\lambda,x}$ is the pointwise limit of a sequence of continuous functions, and is therefore Baire class 1.
As this is true for any $(\lambda,x)\in \mathbb{Q}\times \Delta$, convolution is itself Baire class 1.
\end{proof}

The sequence in the statement of our main theorem will be a defining sequence for one of the types in $\mathcal{T}$.
We will eventually prove an inequality for the type and then show that a similar inequality holds for the spreading model associated with the sequence, but first we need to know under what circumstances a type's defining sequence even has a spreading model.
We already know that a defining sequence $(x_n)_{n=1}^\infty$ for a type $\sigma$ is bounded in norm, but we want to put a condition on $\sigma$ that guarantees $(x_n)_{n=1}^\infty$ is eventually bounded away from zero in norm.
This motivates our next definition.

\begin{defi}
A type $\sigma\in \mathcal{T}$  is called \emph{admissible} if $\sigma(1,0)>\inf_{t>0}\omega_\mathrm{Id}(t)$.
\end{defi}

Note that if $\sigma$ is an admissible type and $(x_n)_{n=1}^\infty$ is a defining sequence for $\sigma$, then $\liminf_{n}\omega_\mathrm{Id}(\|x_n\|)\geq \lim_{n}d(x_n,0)=\sigma(1,0)>\inf_{t>0}\omega_\mathrm{Id}(t)$.  
Thus, since $\omega_\mathrm{Id}$ is an increasing function, we can find $\delta>0$ such that $(x_n)_{n=1}^\infty$ is eventually $\delta$-bounded in norm away from zero.
 From this point forward, we will let $\gamma=\inf_{t>0}\omega_\mathrm{Id}(t)$.

\begin{rem} 
If $\mathrm{Id}\colon (X,\|\cdot\|)\to(X,d)$ is uniformly continuous, then $\gamma=0$.
If, in addition, $d$ is a metric, then the inequality in our definition is trivial, and every nontrivial type will be admissible.
Given our assumption that $d$ is coarsely equivalent to $\|\cdot\|$, we do not need to place any additional conditions on a type to guarantee its defining sequences to be bounded in norm.  Had this not been the case, we would have had to include such a condition in our definition of admissibility.
One condition we could use would be to require a type $\sigma$ to also satisfy the inequality $\sigma(1,0)<\sup_{t<\infty}\rho_\mathrm{Id}(t)$ (a trivial inequality in our case).
In \cite{Raynaud1983}, where the author is concerned with an invariant stable metric $d$ \emph{uniformly} equivalent to the metric induced by $\|\cdot\|$, the author does exactly this.
\end{rem}

At this point, we have established a condition to put on a type to guarantee its defining sequences are bounded in norm and eventually bounded away from zero in norm.
In our goal to obtain a spreading model, we now need an extra condition which will guarantee that a type's defining sequences contain no norm-Cauchy subsequences.

\begin{defi}
We say that a type $\sigma$ is \emph{symmetric} if $\sigma=(-1)\cdot \sigma$, i.e., if $\sigma (\lambda ,x)=\sigma(-\lambda, x)$, for all $(\lambda,x)\in \mathbb{Q}\times \Delta$. Let $\mathcal{S}=\{\sigma\in \mathcal{T}\mid \sigma\text{ is symmetric}\}$ and let $\mathcal{S}_M=\mathcal{S}\cap\mathcal{T}_M$.
\end{defi}

Note that by Lemma \ref{rightcont}, $\mathcal{S}$ is closed, and therefore $\mathcal{S}_M$ is compact for all $M\in\mathbb{R}_+$.

\begin{prop}\label{symadmnonrel}
Say $\sigma\in\mathcal{T}$ is an admissible symmetric type and $(x_n)_{n=1}^\infty$ is a defining sequence for $\sigma$. Then $(x_n)_{n=1}^\infty$ has no $\|\cdot\|$-Cauchy subsequence. 
\end{prop}

\begin{proof}
Suppose to the contrary, that $(x_n)_{n=1}^\infty$ has a $\|\cdot\|$-Cauchy subsequence. After taking this subsequence, we can assume  that $(x_n)_{n=1}^\infty$ converges in norm to some $x\in X$. Then, as $\sigma$ is symmetric, we have that

\begin{align*}
\liminf_{n}d(\lambda x_n,&-\lambda x_n)\\
&= \liminf_{n}\Big(d(\lambda x_n,-\lambda x_n)-\sigma(\lambda,-\lambda x_n)+\sigma(-\lambda,-\lambda x_n)\Big)\\
&= \liminf_{n}\lim_{m}\Big(d(\lambda x_n,-\lambda x_n)-d(\lambda x_m,-\lambda x_n)+d(-\lambda x_m,-\lambda x_n)\Big)\\
&\leq \liminf_{n}\lim_{m}\Big(d(\lambda x_n,\lambda x_m)+d(-\lambda x_m,-\lambda x_n)\Big)\\
&\leq 2\cdot \liminf_{n}\liminf_{m}\omega_\text{Id}(|\lambda|\cdot\|x_n-x_m\|)\\
&= 2 \gamma,
\end{align*}\hfill

\noindent for all $\lambda \in \mathbb{Q}$.
This gives us that $\rho_{\mathrm{Id}}(\|\lambda x\|)\leq\liminf_n\rho_{\mathrm{Id}}(2\|\lambda x_n\|)\leq 2\gamma$, for all $\lambda\in \Q$. As $d$ is coarsely equivalent to the norm of $X$, this can only happen if $x=0$. But then the admissibility of $\sigma$ yields

\[\gamma <\sigma(1,0)=\lim_{n}d(x_n,0)\leq \liminf_{n} \omega_\mathrm{Id}(\|x_n\|)= \gamma,\]\hfill

\noindent a contradiction.
\end{proof}

 \section{Conic classes.}\label{SectionConicClass}
 
 To show the existence of a type that satisfies an $\ell_p$-inequality, we will use a limiting argument and the existence of a shared point of continuity for every finite combination of convolutions and dilations by a scalar.  The definition of conic class below is motivated by the desire to use Lemma \ref{bairelemma2} with the Baire category theorem to find a shared point of continuity, and the need for a minimality argument to make sure this point can be used in the limiting argument.

 \begin{defi}
 A nonempty subset $\mathcal{C}$ of $\mathcal{S}$  is called a \emph{conic class} if \\
 
 \begin{enumerate}[(i)]
 \item $\mathcal{C}\neq\{\overline{0}\}$,
 \item  $\lambda\cdot \sigma \in \mathcal{C}$ for all $\lambda\in \Q$ and $\sigma\in \mathcal{C}$, and
 \item $\sigma *\tau\in \mathcal{C}$ for all $\sigma,\tau\in \mathcal{C}$.
 \end{enumerate}\hfill
 
 \noindent Moreover,  $\mathcal{C}$ is called \emph{admissible} if $\mathcal{C}$ contains an admissible type, i.e., if there exists $\sigma \in \mathcal{C}$ such that $\sigma(1,0)>\gamma$. 
 \end{defi}

 \begin{lemma}
 The set $\mathcal{S}$ is a closed admissible conic class.
 \end{lemma}
 
 \begin{proof}
 That $\mathcal{S}$ is closed follows from Lemma \ref{rightcont}.  The properties (ii) and (iii) follow easily from the definitions of dilation and convolution and from the invariance of $d$.
 All that remains is to show that there is an admissible (and therefore nontrivial) type $\sigma$ in $\mathcal{S}$.
 Let $R<\infty$ be such that $\rho_{\mathrm{Id}}(t)>\gamma$ whenever $t\geq R$.
 By the infinite-dimensionality of $X$, there is a bounded $R$-separated sequence $(x_n)_{n=1}^\infty$ in $(X,\|\cdot\|)$.
 After possibly taking a subsequence, we may suppose that $(x_n)_{n=1}^\infty$ is a defining sequence for some $\sigma\in \mathcal{T}$.
 In this case,
 
 \begin{align*}
 (\sigma \ast (-1)\cdot \sigma) (1,0) &= \lim_{n}\lim_{m}d(x_n-x_m,0)\\
 &\geq \inf_{n\neq m}d(x_n-x_m,0)\\
 &\geq \rho_{\mathrm{Id}}(R)\\
 &> \gamma
 \end{align*}\hfill
 
\noindent That is, the symmetric type $\sigma \ast (-1)\cdot \sigma$ is admissible.
 Therefore $\mathcal{S}$ is a closed admissible conic class.
 \end{proof}

 Now that the existence of a closed admissible conic class has been shown, we will show the existence of one that is minimal (with respect to set inclusion), using the following lemma.
 
 \begin{lemma}
 \label{minimallemma}
 Let $\sigma$ be an admissible type.  Given any $0\leq r_1<r_2$, there is $\alpha\in \mathbb{Q}_+$ such that $\rho_{\mathrm{Id}}(r_1)\leq \alpha \cdot \sigma(1,0)\leq \omega_{\mathrm{Id}}(r_2)$.
 \end{lemma}
 
  \begin{proof}
  Let $(x_n)_{n=1}^\infty$ be a defining sequence for $\sigma$.  
  The admissibility of $\sigma$ implies that $(x_n)_{n=1}^\infty$ is a $\|\cdot\|$-bounded sequence which is eventually $\|\cdot\|$-bounded away from 0.  
  Thus, we may suppose after possibly taking a subsequence that $\lim_{n}\|x_n\|$ exists and is nonzero.
   Let $\alpha\in \mathbb{Q}_+$ be such that $\lim_{n}\|\alpha x_n\| \in [r_1, r_2]$.
   As $\alpha\cdot \sigma(1,0)=\lim_{n}d(\alpha x_n,0)$, we then have 

   \[\rho_{\mathrm{Id}}(r_1)\leq \alpha \cdot \sigma(1,0)\leq \omega_{\mathrm{Id}}(r_2).\]
  \end{proof}

 \begin{prop}\label{minimal}
 Every closed admissible conic class contains a minimal closed admissible conic class.
 \end{prop}

 \begin{proof}
 Fix a closed admissible conic class $\mathcal{C}$. 
 Let $\mathcal{F}$ be the family of closed admissible conic classes contained in $\mathcal{C}$ ordered by reverse set inclusion and let $\{\mathcal{C}_i\}_{i\in I}$ be some chain in $\mathcal{F}$.\\
 
\noindent  \textbf{Claim:} $\bigcap_{i\in I} \mathcal{C}_i$ is a closed admissible conic class.\\
 
 Certainly, $\bigcap_{i\in I} \mathcal{C}_i\subseteq \mathcal{S}$ is closed and satisfies conditions (ii) and (iii) in the definition of conic class. So we only need to show that $\bigcap_{i\in I} \mathcal{C}_i$ contains an admissible type. For that, fix $R<\infty$ such that $\rho_{\mathrm{Id}}(t)>\gamma$ whenever $t\geq R$ and let $\mathcal{B}_i=C_i\cap (\mathcal{T}_{\omega_\text{Id}(R+1)}\setminus \mathrm{int}\mathcal{T}_{\rho_\text{Id}(R)})$ for all $i\in I$. By Lemma \ref{T_Mcompact}, $\mathcal{B}_i$ is compact. Given $i\in I$, let $\sigma_i\in \mathcal{C}_i$ be admissible. By the previous lemma, there is $\alpha_i\in \mathbb{Q}_+$ such that $\alpha_i\cdot \sigma_i\in\mathcal{B}_i$, so $\mathcal{B}_i$ is nonempty. Hence, $\{\mathcal{B}_i\}_{i\in I}$ is a family of compact sets with the finite intersection property, which gives us that  $\bigcap_{i\in I} \mathcal{B}_i\subseteq \bigcap_{i\in I}\mathcal{C}_i$ is nonempty. By our choice of $R$, $\bigcap_{i\in I}\mathcal{B}_i$ can only contain admissible types, hence $\bigcap_{i\in I}\mathcal{C}_i$ contains an admissible type, and the claim is proved.
 
As $\bigcap_{i\in I} \mathcal{C}_i$ is a closed admissible conic class, it is an upper bound for the chain $\{\mathcal{C}_i\}_{i\in I}$ in $\mathcal{F}$. By Zorn's lemma,  $\mathcal{F}$ has a maximal element. That is, $\mathcal{C}$ contains a minimal closed admissible conic class.
\end{proof}
 
 We come now to the main result of this section.  For Raynaud, it was enough to show that the maps $\sigma\mapsto \sigma \ast \alpha\cdot \sigma$, where $\alpha$ is any scalar, share a point of continuity.  He then uses this to show $\sigma \ast \alpha \cdot \sigma = (1+|\alpha|^p)^{1/p}\cdot \sigma$ for some $p\in [1,\infty )$.  With this equality, one may then easily show that for any finite sequence of scalars $\overline{\alpha}=(\alpha_j)_{j=1}^N$, one has $\bigast_{j=1}^N\alpha_j\cdot \sigma=\|\overline{\alpha}\|_p\cdot \sigma$.  In our case however, we will only be able to show that given any $(t_m)_{m=1}^\infty\subseteq \mathbb{Q}$ converging to $\|\overline{\alpha}\|_p$ and $(\lambda, x)\in \mathbb{Q}\times \Delta$, $\limsup_m|\bigast_{j=1}^N\alpha_j\cdot \sigma(\lambda,x)-t_m\cdot \sigma (\lambda,x)|\leq L$, for some constant $L$ depending on $\gamma$.  The next lemma will allow us to make sure $L$ does not depend on the length of $\overline{\alpha}$.
 
\begin{lemma}
\label{commonpoint}
Let $\mathcal{C}$ be a closed admissible conic class.  Then there is an admissible $\phi\in \mathcal{C}$ such that $\phi$ is a common point of continuity for the family of functions $\{\sigma \mapsto \bigast_{j=1}^m \alpha_j\cdot \sigma\ |\ \overline{\alpha}\in \mathbb{Q}^{<\mathbb{N}}\}\subseteq \mathcal{C}^{\mathcal{C}}$.
\end{lemma}

\begin{proof}
By   Lemma \ref{bairelemma2} and Corollary \ref{conv_is_baire} (with $X=\mathcal{C}$ and $\mathcal{F}=\{\sigma\mapsto \alpha\cdot \sigma \ast \beta \cdot \sigma\ |\ \alpha, \beta\in \mathbb{Q}\}$), there is a comeager $G_\delta$ subset $E$ of $\mathcal{C}$ such that $g$ is continuous on $E$ for all $g\in \{\sigma \mapsto \bigast_{j=1}^n \alpha_j\cdot \sigma\ |\ \overline{\alpha}\in \mathbb{Q}^{<\mathbb{N}}\}\subseteq \mathcal{C}^{\mathcal{C}}$.
But $\mathcal{C}$ is closed, and so is locally compact, by Corollary \ref{TM_locallycompact}.
Therefore $E$ is dense in $\mathcal{C}$, by the Baire category theorem, and the statement follows by the admissibility of $\mathcal{C}$.
\end{proof}

\section{Model  associated to an admissible symmetric  type.}\label{SectionModel}

Let $\sigma$ be an admissible symmetric type and  $(x_n)_{n=1}^\infty$ be a defining sequence for $\sigma$.  Then the sequence $(x_n)_{n=1}^\infty$ is bounded, and by Proposition \ref{symadmnonrel}, has no $\|\cdot\|$-Cauchy subsequence. Thus, given a nonprincipal ultrafilter $\mathcal{U}$ on $\N$, we may define a spreading sequence $(\zeta_n)_{n=1}^\infty$ and a spreading model  $S=X\oplus \overline{\text{span}}\{\zeta_n\mid n\in\N\}$ associated to $(x_n)_{n=1}^\infty$ and $\mathcal{U}$ as in Subsection \ref{SubsectionSpreading}.  As in Subsection \ref{SubsectionSpreading}, we let  $(\xi_n)_{n=1}^\infty$ be given by  $\xi_n=\zeta_{2n-1}-\zeta_{2n}$, for all $n\in\N$. 

Let $\tau=\sigma \ast (-1)\cdot \sigma$.
As $\sigma=\lim_n \overline{x}_n$, we may assume after taking a subsequence that $\tau=\lim_n \overline{y}_n$ where $y_n=x_{2n-1}-x_{2n}$.
As $(x_n)_{n=1}^\infty$ has no $\|\cdot\|$-Cauchy subsequence, we may further assume after taking another subsequence that $\inf_{n\neq m}\|x_n-x_m\|>0$. As $\tau(1,0)=\lim_n d(y_n,0)\geq \rho_{\mathrm{Id}}(\inf_{n\neq m}\|x_n-x_m\|)$, by dilating $\sigma$, we can also assume that $\tau$ is an admissible type.  It is clear that $(\xi_n)_{n=1}^\infty$ is the spreading model of $(y_n)_{n=1}^\infty$ for the ultrafilter $\mathcal{U}$.

From this point forward, we fix a minimal closed admissible conic class $\mathcal{C}$ and an admissible $\phi\in \mathcal{C}$ that is a common point of continuity for the family of functions $\mathcal{F}= \{\sigma \mapsto \bigast_{j=1}^n \alpha_j\cdot \sigma\ |\ n\in \mathbb{N}, \alpha\in \mathbb{Q}^n\}\subseteq \mathcal{C}^{\mathcal{C}}$ such that $\psi=\phi\ast (-1)\cdot \phi$ is admissible (this is possible by Lemma \ref{commonpoint}).
We also fix a defining sequence $(x_n)_{n=1}^\infty$ for $\phi$ with unique (see Section \ref{SubsectionSpreading}) spreading model $(S,\tn\cdot\tn)$ such that $y_n=x_{2n-1}-x_{2n}$ is a defining sequence for $\psi$.
We will let $(\zeta_n)_{n=1}^\infty$ be the spreading sequence associated with $S$ and $\xi_n=\zeta_{2n-1}-\zeta_{2n}$ for every $n\in\mathbb{N}$.

\begin{defi}
Given $(\alpha_j)_{j=1}^m \in \mathbb{Q}^{<\mathbb{N}}$, we say that $\sum_{j=1}^m\alpha_j\zeta_j$ \emph{realizes} the type $\bigast_{j=1}^m\alpha_j \cdot \phi$. 
\end{defi}

\begin{remark}\label{remrealizes}
Notice that, if $u=\sum_{j=1}^{m_1}\alpha_j\zeta_j$ realizes $\sigma$, and $v=\sum_{j=m_1+1}^{m_2}\beta_j\zeta_j$ realizes $\tau$, it follows that $u+v$ realizes $\sigma*\tau$.
\end{remark}

\subsection{Basic properties of  $\tn \cdot\tn $.} We will now prove some technical lemmas which will be important in the proof of the main theorem of these notes.  The main lemma here is Lemma \ref{N(u-v)ext2}, which will allow us to derive inequalities involving the type $\psi$.

\begin{lemma}\label{N(u-v)}
Say $u\neq v\in \mathrm{span}_\mathbb{Q}\{\zeta_n\mid n\in\N\}$ realize $\sigma$ and $\tau$, respectively.  Then for every $(\lambda, x)\in \mathbb{Q}\times \Delta$,

\[\sup_{0<\eps\leq |\lambda|\tn  u- v\tn}\rho_\mathrm{Id}(|\lambda |\tn  u- v\tn -\eps )\leq\sigma(\lambda,x)+\tau(\lambda,x)\]\hfill

\noindent and

\[|\sigma(\lambda,x)-\tau(\lambda,x)|\leq \inf_{\eps>0}\omega_\mathrm{Id}(|\lambda| \tn u-v \tn +\eps).\]\hfill

\noindent In particular, we have for each $\delta>0$ that 
 \begin{enumerate}[(i)] 
  \item $\tn  u\tn  >\delta$ implies $\sigma(1,0)\geq \rho_\mathrm{Id}(\delta)$, and
  \item $\sigma(1,0)>\omega_{\mathrm{Id}}(\delta)$ implies $\tn  u\tn  \geq\delta$.
 \end{enumerate}

\end{lemma}

\begin{proof}
Say $u=\sum_{j=1}^m\alpha_j\zeta_j$ and $v=\sum_{j=1}^m\beta_j\zeta_j$, for some $(\alpha_j)_{j=1}^m,(\beta_j)_{j=1}^m\in\mathbb{Q}^{<\mathbb{N}}$. Then

\begin{align*}
\rho_\mathrm{Id}(|\lambda| \tn u-v \tn -\eps )
&\leq \limsup_{n_m}\dots \limsup_{n_1} \rho_\mathrm{Id}\Big(\Big\|\lambda \sum_{j=1}^m(\alpha_j-\beta_j)x_{n_j}\Big\| \Big)\\
&\leq \lim_{n_m}\ldots\lim_{ n_1}d\Big(\lambda \sum_{j=1}^m(\alpha_j-\beta_j)x_{n_j},0\Big)\\
&\leq\lim_{n_m}\ldots\lim_{ n_1}\Big(d\Big(\lambda\sum_{j=1}^m\alpha_j x_{n_j},x\Big)+d\Big(\lambda\sum_{j=1}^m\beta_j x_{n_j},x\Big)\Big)\\
&= \sigma(\lambda,x)+\tau(\lambda,x)
\end{align*}\hfill

\noindent for every $0<\eps< |\lambda|\tn u-v \tn$.  Similarly,

\begin{align*}
|\sigma(\lambda,x)-\tau(\lambda,x)|
&=\lim_{n_m}\ldots\lim_{ n_1}\Big|d\Big(\lambda\sum_{j=1}^m\alpha_j x_{n_j},x\Big)-d\Big(\lambda\sum_{j=1}^m\beta_j x_{n_j},x\Big)\Big|\\
&\leq \lim_{n_m}\ldots\lim_{ n_1}d\Big(\lambda \sum_{j=1}^m(\alpha_j-\beta_j)x_{n_j},0\Big)\\
&\leq\liminf_{n_m}\dots \liminf_{n_1} \omega_\mathrm{Id}\Big(\Big\|\lambda \sum_{j=1}^m(\alpha_j-\beta_j)x_{n_j}\Big\| \Big)\\
&\leq \omega_\mathrm{Id}(|\lambda| \tn u-v \tn +\eps )
\end{align*}\hfill

\noindent for all $\eps >0$.  The particular case follows by letting $v=0$ and $\lambda=1$.
\end{proof}

Let $H=\text{span}_\mathbb{Q}\{\xi_i\mid i\in\N\}\subseteq S$. Given $\overline{\alpha}=(\alpha_j)_{j=1}^m\in \mathbb{Q}^{<\mathbb{N} }$, we define a bounded linear map $T_{\overline{\alpha}} \colon \overline{H}\to \overline{H}$ as follows. For each $n\in\N$ let

\[T_{\overline{\alpha}} (\xi_n)=\sum_{j=1}^m \alpha_j \xi_{mn+j-1}\]\hfill

\noindent and extend $T_{\overline{\alpha}}$ linearly to $H$.  As $(\xi_n)_{n=1}^\infty$ is $1$-spreading, we have that $\tn T_{\overline{\alpha}}(u)\tn\leq \|\alpha\|_1 \tn u\tn$, for all $u\in H$. Hence, we can extend $T_{\overline{\alpha}}$ to a bounded operator $T_{\overline{\alpha}}\colon \overline{H}\to \overline{H}$. If $\overline{\alpha}=(\alpha_1)$ is a sequence of length 1, then  $T_{\overline{\alpha}}u$ is just the scaling of $u$ by $\alpha_1$.  We also  define the function $\widehat{T}_{\overline{\alpha}} \colon \mathcal{C}\to \mathcal{C}$ by letting

\[\widehat{T}_{\overline{\alpha}} (\sigma)=\bigast_{j=1}^m\alpha_j\cdot \sigma\]\hfill

\noindent for all $\sigma\in \mathcal{C}$.

\begin{lemma}
\label{circk}
Let $\overline{\alpha}=(\alpha_i)_{i=1}^n,\overline{\beta}=(\beta_j)_{j=1}^m\in\mathbb{Q}^{<\mathbb{N}}$. Let  $\overline{\gamma}=(\gamma_k)_{k=1}^{nm}\in \mathbb{Q}^{<\mathbb{N}}$, where $\gamma_k=\alpha_i \beta_j$ whenever $k=n(j-1)+i$. Then $T_\alpha \circ T_\beta =T_{\gamma}$ and $\widehat{T}_\alpha \circ \widehat{T}_\beta =\widehat{T}_{\gamma}$. 
\end{lemma}

\begin{proof}
For any $k\in \mathbb{N}$,

\begin{align*}
(T_{\overline{\alpha}} \circ T_{\overline{\beta}} )(\xi_k)& =T_{\overline{\alpha}} \Big(\sum_{j=1}^m\beta_j \xi_{mk+j-1}\Big)\\ 
&=\sum_{j=1}^m \sum_{i=1}^n\alpha_i \beta_j \xi_{n(mk+j-1)+i-1}\\
&=\sum_{j=1}^m \sum_{i=1}^n \alpha_i\beta_j \xi_{nmk+n(j-1)+i-1}\\ &=  \sum_{\ell=1}^{nm}\gamma_\ell \xi_{nmk+\ell-1}\\
&=T_{\overline{\gamma}}(\xi_k)
\end{align*}\hfill

\noindent Therefore $T_{\overline{\alpha}} \circ T_{\overline{\beta}}=T_{\overline{\gamma}}$, by linearity and continuity. Similarly,

\begin{align*}
(\widehat{T}_{\overline{\alpha}}\circ\widehat{T}_{\overline{\beta}})(\sigma)
=\widehat{T}_\alpha \Big(\bigast_{j=1}^m \beta_j\sigma\Big)
=\bigast_{j=1}^m\bigast_{i=1}^n\alpha_i\beta_j\sigma=\bigast_{\ell=1}^{nm} \gamma_\ell \sigma
=\widehat{T}_{\overline{\gamma}}(\sigma).
\end{align*}\hfill

\noindent for all $\sigma\in \mathcal{C}$, and so $\widehat{T}_{\overline{\alpha}}\circ\widehat{T}_{\overline{\beta}}=\widehat{T}_{\overline{\gamma}}$.
\end{proof}

The previous lemma justifies the following definition.

\begin{defi}
Let $\overline{\alpha}=(\alpha_i)_{i=1}^n,\overline{\beta}=(\beta_j)_{j=1}^m\in\mathbb{Q}^{<\mathbb{N}}$.  We define $\overline{\alpha} \circ \overline{\beta}=(\gamma_k)_{k=1}^{nm}\in \mathbb{Q}^{<\mathbb{N}}$ by $\gamma_k=\alpha_i \beta_j$ whenever $k=n(j-1)+i$.  We define $\overline{\alpha}^{\circ k}$ recursively by letting $\overline{\alpha}^{\circ 1}=\overline{\alpha}$ and $\overline{\alpha}^{\circ k+1}=\overline{\alpha}\circ \overline{\alpha}^{\circ k}$ for every $k\in \mathbb{N}$.
\end{defi}

\begin{remark}
Notice that, $\widehat{T}^k_{\overline{\alpha}}=\widehat{T}_{\overline{\alpha}^{\circ k}}$ for all $\overline{\alpha}\in\Q^{<\N}$ and all $k\in\N$.
\end{remark}

\begin{lemma}\label{tauk}
Let ${\overline{\alpha}}=(\alpha_j)_{j=1}^m\in \Q^{<\mathbb{N}}$. Say $u\in H$ realizes the type $\sigma$. Then,  $T_{\overline{\alpha}}(u)$ realizes $\widehat{T}_{\overline{\alpha}}(\sigma)$. 
\end{lemma}

 \begin{proof}
Suppose $u=\sum_{i=1}^{n}\lambda_i\xi_i$, so $\sigma= \bigast_{i=1}^n\lambda_i \cdot\psi$.
 Then 
 
 \[T_{\overline{\alpha}}(u)=\sum_{i=1}^n\lambda_i \sum_{j=1}^m\alpha_j \xi_{mi+j-1}=\sum_{j=1}^m\sum_{i=1}^n\alpha_j\lambda_i\xi_{mi+j-1}\]\hfill
 
 \noindent which realizes the type 
 
 \[\bigast_{i=1}^m \bigast_{j=1}^n\alpha_j\lambda_i\cdot \psi=\bigast_{j=1}^m\alpha_j\cdot \bigast_{i=1}^n\lambda_i\cdot \psi=\widehat{T}_{\overline{\alpha}} (\sigma).\]
 \end{proof}

\begin{lemma}\label{N(u-v)ext2}
Fix $N\in\mathbb{N}$.  For each $1\leq i\leq N$, suppose $u_i, v_i\in H$ realize $\sigma$ and $\tau$, respectively. Let $(\overline{\alpha}_i)_{i=1}^N, (\overline{\beta}_i)_{i=1}^N\subseteq \mathbb{Q}^{<\mathbb{N}}$ and $(b_i)_{i=1}^N\in\Q^N$. Then for every $(\lambda, x)\in \mathbb{Q}\times \Delta$, we have that

\[\Big|\bigast_{i=1}^Nb_i\cdot\widehat{T}_{\overline{\alpha}_i}\sigma(\lambda,x)-\bigast_{i=1}^Nb_i\cdot\widehat{T}_{\overline{\beta}_i}\tau(\lambda,x)\Big|\leq \inf_{\eps>0}\omega_\mathrm{Id}\Big(|\lambda| \sum_{i=1}^N|b_i|\cdot\tn T_{\overline{\alpha}_i}u_i-T_{\overline{\beta}_i}v_i \tn +\eps\Big).\]
\end{lemma}

\begin{proof}
For each $m\in \mathbb{N}$, let $s_m\colon H\to H$ be the linear map given by  $s_m(\xi_n)=\xi_{n+m}$ for each $n\in \mathbb{N}$. We construct sequences $(u_i')_{i=1}^N, (v_i')_{i=1}^N\subseteq H$ recursively as follows.  Let $u_1'=b_1T_{\overline{\alpha}_1}u_1$ and $v_1'=b_1 T_{\overline{\beta}_1}v_1$.  Given $u_i', v_i'$ for some $1\leq i<N$, let $m_i=\max \{\mathrm{supp}(u_i') \cup \mathrm{supp}(v_i')\}$ and then let $u_{i+1}'=b_{i+1}s_{m_i}(T_{\overline{\alpha}_{i+1}}u_{i+1})$ and $v_{i+1}'=b_{i+1}s_{m_i}(T_{\overline{\beta}_{i+1}}v_{i+1})$.  Clearly, both sequences  $(u_i')_{i=1}^N$ and   $(v_i')_{i=1}^N$ have  disjoint supports. Hence, by Lemma \ref{tauk} and Remark \ref{remrealizes}, $\sum_{i=1}^Nu_i'$ and $\sum_{i=1}^Nv_i'$ realize $\bigast_{i=1}^Nb_i\cdot\widehat{T}_{\overline{\alpha}_i}\sigma$ and   $\bigast_{i=1}^Nb_i\cdot\widehat{T}_{\overline{\beta}_i}\tau$, respectively.  Thus, by Lemma \ref{N(u-v)} and the fact that   $(\xi_n)_{n=1}^\infty$ is $1$-spreading, we have that

\begin{align*}
\Big|\bigast_{i=1}^Nb_i\cdot\widehat{T}_{\overline{\alpha}_i}\sigma(\lambda,x)-\bigast_{i=1}^Nb_i\cdot\widehat{T}_{\overline{\beta}_i}\tau(\lambda,x)\Big|
&\leq \inf_{\eps>0}\omega_\mathrm{Id}\Big(|\lambda|\cdot \tn \sum_{i=1}^N (u_i'-v_i') \tn +\eps \Big)\\
&\leq \inf_{\eps>0}\omega_\mathrm{Id}\Big(|\lambda| \sum_{i=1}^N\tn u_i'-v_i'\tn +\eps\Big)\\
&=\inf_{\eps>0}\omega_\mathrm{Id}\Big(|\lambda| \sum_{i=1}^N|b_i|\cdot\tn T_{\overline{\alpha}_i}u_i-T_{\overline{\beta}_i}v_i\tn +\eps\Big).
\end{align*}
\end{proof}

\section{Coarse approximating sequences.}\label{SectionCoarseApp}

The goal of this section is to show that the type $\psi$ satisfies the conclusion of Proposition \ref{theoremspecialtype} below. For that, we introduce the notion of coarse approximating sequences.

\begin{defi}
Let $u=\sum_{i=1}^k\alpha_i\xi_i\in\text{span}\{\xi_n\mid n\in\N\}$. We say that a vector $v\in \text{span}\{\xi_n\mid n\in\N\}$ is a \emph{spreading} of $u$ if $v=\sum_{i=1}^k\alpha_i\xi_{n_i}$ for some  $n_1<\ldots<n_k\in\N$.
\end{defi}

\begin{defi}\label{defapptypes}
Let  $(\overline{\alpha}_i)_{i=1}^N\subseteq\Q^{<\mathbb{N}}$ and $(\beta_i)_{i=1}^N\in \mathbb{R}_+^N$. A sequence of types $(\sigma_{n})_{n=1}^\infty\subseteq \mathcal{C}$  is called  a \emph{coarse $(\overline{\alpha}_i,\beta_i)_{i=1}^N$-approximating sequence} if there exists a sequence $(u_{n})_{n=1}^\infty\subseteq H$ and sequences $(u_{i,n})_{n=1}^\infty\subseteq H$ for each $1\leq i\leq N$ such that \\

\begin{enumerate}[(i)]
\item   $u_{n}$ realizes $\sigma_{n}$ for all $n\in\N$,
\item  $u_{i,n}$ is a spreading of $u_n$  for each $n\in \mathbb{N}$ and $1\leq i\leq N$,
 and
\item  $\lim_n\tn  T_{{\overline{\alpha}_i}}(u_{n})-\beta_i u_{i,n}\tn  =0$ for all $1\leq i\leq N$.
\end{enumerate}\hfill

\end{defi}

\begin{lemma}
Suppose $\overline{\alpha}\in \mathbb{Q}^{<\mathbb{N}}$, $\beta\geq 0$, and $(u_n)_{n=1}^\infty \subseteq H$.  If there is a spreading $(u_n')_{n=1}^\infty$ of $(u_n)_{n=1}^\infty$ such that $\lim_n\tn T_{\overline{\alpha}} (u_n)-\beta u_n'\tn=0$, then for every $k\in \mathbb{N}$ there is a spreading $(u_n'')_{n=1}^\infty$ of $(u_n)_{n=1}^\infty$ such that $\lim_n\tn T_{\overline{\alpha}}^k (u_n)-\beta^k u_n''\tn=0$.
\end{lemma}

\begin{proof}
For $k=1$ the  result is trivial. Suppose the result holds for some $k\in\N$. Let $(u_n'')_{n=1}^\infty$ be a spreading of $(u_n)_{n=1}^\infty$ such that $\lim_n\tn T_{\overline{\alpha}}^k (u_n)-\beta^k u_n''\tn=0$. 
By the definition of $T_{\overline{\alpha}}$, it follows that $(T_{\overline{\alpha}}(u_n''))_{n=1}^\infty$ is a spreading of $(T_{\overline{\alpha}}(u_n))_{n=1}^\infty$, so there exists a spreading $(u_n''')_{n=1}^\infty$ of $(u_n)_{n=1}^\infty$ such that also $(T_{\overline{\alpha}} (u_n'')-\beta u_n''')_{n=1}^\infty$ is a spreading of $(T_{\overline{\alpha}} (u_n)-\beta u_n')_{n=1}^\infty$.  Thus, by the 1-equivalence of $(\xi_n)_{n=1}^\infty$ with all its subsequences,

\begin{align*}
\tn T_{\overline{\alpha}}^{k+1} (u_n)-\beta^{k+1} u_n'''\tn 
&\leq \tn T_{\overline{\alpha}}^{k+1} (u_n)- T_{\overline{\alpha}}(\beta^k u_n'')\tn +\tn T_{\overline{\alpha}}(\beta^k u_n'')-\beta^{k+1} u_n'''\tn\\
&=\tn T_{\overline{\alpha}}(T_{\overline{\alpha}}^k (u_n)-\beta^k u_n'')\tn + \beta^k \tn T_{\overline{\alpha}}(u_n'')-\beta u_n'''\tn\\
&\leq  \tn T_{\overline{\alpha}}\tn\cdot\tn T_{\overline{\alpha}}^k (u_n)-\beta^k u_n''\tn + \beta^k \tn T_{\overline{\alpha}}(u_n)-\beta u_n'\tn.
\end{align*}\hfill

\noindent Therefore $\lim_n \tn T_{\overline{\alpha}}^{k+1} (u_n)-\beta^{k+1} u_n'''\tn=0$, so the result holds for $k+1$. By induction, we are finished.
\end{proof}

With the above lemma and Lemma \ref{circk}, we have the following corollary.

\begin{cor}\label{alphabetak}
If $(\sigma_n)_{n=1}^\infty$ is a coarse $({\overline{\alpha}}_i, \beta_i)_{i=1}^N$-approximating sequence, then it is also a coarse $({\overline{\alpha}}_i^{\circ k_i},\beta_i^{k_i})_{i=1}^N$-approximating sequence for every $k_1,\dots,k_N\in \mathbb{N}$.
\end{cor}

\begin{lemma}\label{case3}
Suppose $({\overline{\alpha}}_i)_{i=1}^N\subseteq \mathbb{Q}^{<\mathbb{N}}$ is such that ${\overline{\alpha}}_i\circ {\overline{\alpha}}_j={\overline{\alpha}}_j\circ{\overline{\alpha}}_i$ for all $1\leq i,j\leq N$.  Then there are $(\beta_i)_{i=1}^N\in \mathbb{R}^N$ and $(\sigma_n)_{n=1}^\infty\subseteq \mathcal{C}$ such that  $(\sigma_n)_{n=1}^\infty$ is a coarse $({\overline{\alpha}}_i,\beta_i)_{i=1}^N$-approximating sequence and $\beta_i\in [\|{\overline{\alpha}}_i\|_\infty, \|{\overline{\alpha}}_i\|_1]$ for each $1\leq i\leq N$. Moreover, we may choose $(\sigma_n)_{n=1}^\infty$ so that for all $n\in \mathbb{N}$,  $b_1\leq \sigma_n(1,0)\leq b_2$ for some $\gamma <b_1\leq b_2$ not depending on $n$. 
\end{lemma}

\begin{proof}
For those ${\overline{\alpha}}_i$'s that are length 1 sequences, the proposition is clear with $\{\beta_i\}={\overline{\alpha}}_i$.
So suppose for each $1\leq i\leq N$ that ${\overline{\alpha}}_i$ is a sequence of length at least 2. 
As the basis $(\xi_n)_{n=1}^\infty$ of $\overline{H}$ is $1$-unconditional and $1$-spreading, we have that $\|{\overline{\alpha}}_i\|_\infty \tn u\tn \leq \tn  T_{{\overline{\alpha}}_i}(u)\tn  \leq \|{\overline{\alpha}}_i\|_1\tn u\tn $, for all $u\in \overline{H}$ and all $1\leq i\leq N$. Also, for each $1\leq i\leq N$, it is clear from the definition of $T_{\overline{\alpha}_i}$ that $\tn  T_{\overline{\alpha}_i}(u)-\xi_1\tn  >0$ for all $u\in \overline{H}$, and so $T_{{{\overline{\alpha}}}_i}$ is not invertible. Hence, the spectrum of $T_{{\overline{\alpha}}_i}$ has a  real non-negative boundary point, and so  $T_{{\overline{\alpha}}_i}$ has a real non-negative approximate eigenvalue for each $1\leq i\leq N$ by \cite[Proposition IV.1]{KrivineMaurey1981}. By Lemma \ref{circk}, $T_{{\overline{\alpha}}_i}$ commutes with $T_{{\overline{\alpha}}_j}$ for all $1\leq i,j\leq N$.  Thus, by \cite[Proposition 12.18]{BenyaminiLindenstrauss}, there exists $(\beta_i)_{i=1}^N\in \mathbb{R}_+^N$ and a single normalized sequence $(u_n)_{n=1}^\infty\subseteq \overline{H}$ such that $\lim_n\tn T_{{\overline{\alpha}}_i}u_n-\beta_i u_n\tn=0$ for every $1\leq i\leq N$.  As $\tn u_n\tn=1$ for each $n\in \mathbb{N}$, the bounds above for $\tn T_{{\overline{\alpha}}_i}(u)\tn$ yield that $\beta_i\in [\|{\overline{\alpha}}_i\|_\infty,\|{\overline{\alpha}}_i\|_1]$ for each $1\leq i\leq N$.  By density, one may assume that $(u_n)_{n=1}^\infty\subseteq H$ and $1\leq \tn u_n\tn \leq 2$ for all $n\in \mathbb{N}$. Finally, let $\delta>0$ be such that $\rho_{\text{Id}}(\delta/2)>\gamma$ and let $\sigma_n$ be the type realized by $\delta u_n$ for each $n\in \mathbb{N}$. The result now follows by letting $b_1=\rho_{\text{Id}}(\delta)$ and $b_2=\omega_\text{Id}(3\delta)$ (see Lemma \ref{N(u-v)}).
\end{proof}

\begin{lemma}\label{limitofseq}
Suppose $(\overline{\alpha}_i)_{i=1}^N\subseteq \mathbb{Q}^{<\mathbb{N}}$ is such that $\overline{\alpha}_i\circ \overline{\alpha}_j=\overline{\alpha}_j\circ \overline{\alpha}_i$ for all $1\leq i,j\leq N$.  Then there is $(\beta_i)_{i=1}^N\in \mathbb{R}^N$ such that every $\sigma\in \mathcal{C}$ is the limit of a coarse $(\overline{\alpha}_i,\beta_i)_{i=1}^N$-approximating sequence and $\beta_i\in [\|\overline{\alpha}_i\|_\infty,\|\overline{\alpha}_i\|_1]$ for all $1\leq i\leq N$.
\end{lemma}

\begin{proof}
Let $\gamma< b_1\leq b_2$, $(\beta_i)_{i=1}^N\in \mathbb{R}^N$ and  $(\sigma_n)_{n=1}^\infty$ be given by  Lemma \ref{case3}, so that $(\sigma_n)_{n=1}^\infty$ is a coarse $(\overline{\alpha}_i,\beta_i)_{i=1}^N$-approximating sequence  and  $b_1\leq \sigma_n(1,0)\leq b_2$ for every $n\in \mathbb{N}$. Let $\tilde{\mathcal{C}}$ be the subset of $\mathcal{C}$ consisting of all types of $\mathcal{C}$ which are the limit of a coarse $({\overline{\alpha}}_i,\beta_i)_{i=1}^N$-approximating sequence. 
As $\mathcal{T}_{b_1,b_2}\coloneqq\{\sigma\in \mathcal{T}\mid b_1\leq \sigma(1,0)\leq b_2\}$ is compact and metrizable, $(\sigma_n)_n$ has a converging subsequence which converges to an element $\sigma\in \mathcal{C}\cap\mathcal{T}_{b_1,b_2}$. 
A subsequence of a coarse $({\overline{\alpha}}_i,\beta_i)_{i=1}^N$-approximating sequence is still a coarse $({\overline{\alpha}}_i,\beta_i)_{i=1}^N$-approximating sequence, so we have that $\tilde{\mathcal{C}}\neq \{0\}$, and in particular $\tilde{\mathcal{C}}$ contains an admissible type.

By the minimality of $\mathcal{C}$, it is enough to show that $\tilde{\mathcal{C}}$ is a closed conic class.  
Suppose $\sigma\in \tilde{\mathcal{C}}$ and $(\sigma_n)_{n=1}^\infty$ is a coarse $({\overline{\alpha}}_i,\beta_i)_{i=1}^N$-approximating sequence converging to $\sigma$.
Then, by Lemma \ref{rightcont}, $\lambda\cdot \sigma$ is the limit of $(\lambda\cdot \sigma_n)_{n=1}^\infty$, which is easily seen to be a coarse $({\overline{\alpha}}_i,\beta_i)$-approximating sequence for every $\lambda\in \mathbb{Q}$.
Thus $\tilde{C}$ is closed under dilation by any $\lambda\in \mathbb{Q}$.

 Let $D$ be a metric compatible with the topology of $\mathcal{T}$.  Say $\sigma,\tau\in \tilde{\mathcal{C}}$ and let us show that $\sigma*\tau\in \tilde{\mathcal{C}}$. 
 Let $(\sigma_n)_{n=1}^\infty$ and $(\tau_n)_{n=1}^\infty$ be coarse $({\overline{\alpha}}_i,\beta_i)_{i=1}^N$-approximating sequences in $\mathcal{C}$ converging to $\sigma$ and $\tau$, respectively.  
 As the convolution is separately continuous, we have $\lim_k \sigma_k\ast \tau =\sigma\ast \tau$ and, for a fixed $k\in\N$, $\lim_n\sigma_k*\tau_n=\sigma_k*\tau$. For each $k\in\N$, let $n(k)\geq k$ be such that 
 
\[D(\sigma_k*\tau_{n(k)},\sigma_k *\tau)\leq 2^{-k}.\]\hfill
 
 \noindent If we set $\sigma'_k=\sigma_k*\tau_{n(k)}$, then $\lim_k\sigma_k'=\sigma*\tau$. 
 
For each $1\leq i\leq N$, let $(u_{n})_{n=1}^\infty$, $(u_{i,n})_{n=1}^\infty$, $(v_{n})_{n=1}^\infty$  and  $(v_{i,n})_{n=1}^\infty$ be sequences realizing $(\sigma_n)_{n=1}^\infty$ and $(\tau_n)_{n=1}^\infty$ respectively, as given by the definition of  coarse $({\overline{\alpha}}_i,\beta_i)_{i=1}^N$-approximating sequences. By translating the supports of $v_{n(k)}$ and $v_{i,n(k)}$ if necessary, we may assume that $\text{supp}(u_k)<\text{supp}(v_{n(k)})$ and $\text{supp}(u_{i,k})<\text{supp}(v_{{i,n(k)}})$ for all $1\leq i\leq N$ and $k\in\N$.  Let  $(z_{k})_{k=1}^\infty=(u_{k}+v_{n(k)})_{k=1}^\infty$, so  $z_k$ realizes $\sigma'_k$ for each $k\in\N$.  Set $(z_{i,k})_{k=1}^\infty=(u_{i,k}+v_{i,n(k)})_{k=1}^\infty$ for all $1\leq i\leq N$, so $z_{i,k}$ is a spreading of $z_k$. This gives us that  $(\sigma'_k)_{k=1}^\infty$ is a coarse $({\overline{\alpha}}_i,\beta_i)_{i=1}^N$-approximating sequence.  Thus $\sigma\ast \tau \in \tilde{\mathcal{C}}$, and so $\tilde{\mathcal{C}}$ is closed under convolution.
 
 Finally, let us show that $\tilde{\mathcal{C}}$ is closed. Say $(\sigma_k)_{k=1}^\infty$ is a sequence in $\tilde{\mathcal{C}}$ converging to $\sigma\in \mathcal{C}$. For each $k\in \N$, there exists a coarse $({\overline{\alpha}}_i,\beta_i)_{i=1}^N$-approximating sequence $(\sigma_{k,n})_{n=1}^\infty$ in $\mathcal{C}$ converging to $\sigma_k$.
 For each $k\in\N$, let $(u_{k,n})_{n=1}^\infty$ be a sequence realizing $(\sigma_{k,n})_{n=1}^\infty$ and let $(u_{k,i,n})_{n=1}^\infty$ be a spreading of $(u_{k,n})_{n=1}^\infty$ for each $1\leq i\leq N$ as given by Definition \ref{defapptypes}.
  For each $k\in\N$, choose an integer $n(k)\geq k$ such that $D(\sigma_{k,n(k)},\sigma_k)\leq 1/k$ and $\tn T_{{\overline{\alpha}}_i}(u_{k,n(k)})-\beta_i u_{k,i,n(k)}\tn<1/k$ for each $1\leq i\leq N$. 
  Set $\tau_k=\sigma_{k,n(k)}$ for each $k\in\N$. Then $(\tau_k)_{k=1}^\infty$ is a coarse $({{\overline{\alpha}}_i},\beta_i)_{i=1}^N$-approximating sequence converging to $\sigma$.
  That is, $\sigma\in \tilde{\mathcal{C}}$. 
  Thus $\tilde{\mathcal{C}}$ is closed since $\sigma$ was an arbitrary limit point.
  By what was shown, $\tilde{\mathcal{C}}$ is a closed admissible conic class contained in $\mathcal{C}$ and by the minimality of $\mathcal{C}$, we are finished.
\end{proof}

\begin{prop}\label{theoremspecialtype}
Suppose $(\overline{\alpha}_i)_{i=1}^N\subseteq \mathbb{Q}^{<\mathbb{N}}$ is such that $\overline{\alpha}_i\circ \overline{\alpha}_j=\overline{\alpha}_j\circ \overline{\alpha}_i$ for all $1\leq i,j\leq N$.  There exists $\overline{\beta}=(\beta_i)_{i=1}^N\in \R^N$ such that $\beta_i\in [\|\overline{\alpha}_i\|_\infty,\|\overline{\alpha}_i\|_1]$ for all $1\leq i\leq N$ and 

\[\limsup_m \Big|\bigast_{i=1}^Nb_i\cdot \widehat{T}^{k_i}_{\overline{\alpha}_i}\psi (\lambda, x)-\bigast_{i=1}^Nb_i{\beta_{i,m}^{k_i}}\cdot \psi (\lambda,x)\Big|\leq \gamma\] \hfill

\noindent for every  $(b_i)_{i=1}^N\in \Q^N$, every $(k_i)_{i=1}^N\in \N^N$, every $(\lambda, x)\in \mathbb{Q}\times \Delta$, and every sequence $(\overline{\beta}_{m})_{m=1}^\infty\subseteq \mathbb{Q}_+^N$ converging to $\overline{\beta}$, where $\overline{\beta}_{m}=(\beta_{i,m})_{i=1}^N$ for all $m\in\N$. 
\end{prop}

\begin{proof}
Let  $(\beta_i)_{i=1}^N\in \mathbb{R}^N$ be given by Lemma \ref{limitofseq} and let $(\phi_n)_{n=1}^\infty$ be a coarse $(\overline{\alpha}_i,\beta_i)_{i=1}^N$-approximating sequence converging to $\phi$, also given by  Lemma \ref{limitofseq}. For each $n\in\N$ let $\psi_n=\phi_n*(-1)\cdot\phi_n$. Then, by our choice of $\phi$ (see Lemma \ref{commonpoint}) we have that

\[\lim_n\bigast_{i=1}^N b_i\cdot\widehat{T}^{k_i}_{\overline{\alpha}_i}\psi_n (\lambda, x)=\bigast_{i=1}^N b_i\cdot \widehat{T}^{k_i}_{\overline{\alpha}_i}\psi(\lambda, x)\] and \[ \lim_n\bigast_{i=1}^N b_i\beta_{i,m}^{k_i}\cdot \psi_n(\lambda,x)=\bigast_{i=1}^N b_i\beta_{i,m}^{k_i}\cdot \psi (\lambda,x)\]\hfill

\noindent for all $(\lambda,x)\in\Q\times \Delta$ and all $m\in\N$. 

By Corollary \ref{alphabetak}, $(\phi_n)_{n=1}^\infty$ is a coarse $(\overline{\alpha}^{\circ k_i}_i,\beta_i^{k_i})_{i=1}^N$-approximating sequence and we can pick a sequence  $(u_n)_{n=1}^\infty$  realizing $(\phi_n)_{n=1}^\infty$ and sequences $(u_{i,n})_{n=1}^\infty$ which are spreadings of $(u_n)_{n=1}^\infty$ and satisfy $\lim_n\tn T_{\overline{\alpha}^{\circ k_i}}u_n-\beta_i^{k_i}u_{i,n}\tn=0$ for every $1\leq i\leq N$.  For each $n\in \mathbb{N}$, let $u_n'\in H$ have the same basis coordinates as $u_n$ except shifted over so that $\mathrm{supp}(u_n)< \mathrm{supp}(u_n')$ and $\mathrm{supp}(u_{i,n})<\mathrm{supp}(u_n')$ for every $1\leq i\leq N$. 
For each $1\leq i\leq N$ and $n\in \mathbb{N}$, let $u'_{i,n}$ be a spreading of $u_n$ so that $T_{\overline{\alpha}^{\circ k_i}_i}u_n'-\beta^{k_i} u_{i,n}'$ is a spreading of $T_{\overline{\alpha}^{\circ k_i}_i}u_n-\beta^{k_1} u_{i,n}$ and such that $\mathrm{supp}(u_{i,n})<\mathrm{supp}(u_{i,n}')$. 

Notice that  both $u_{n}-u_n'$ and $u_{i,n}-u_{i,n}'$ realize $\psi_n$. Therefore, by Lemma \ref{N(u-v)ext2}, we have that

\begin{align*}
\Big|\bigast_{i=1}^N b_i\cdot &\widehat{T}^{k_i}_{\overline{\alpha}_i} \psi_n (\lambda, x)-\bigast_{i=1}^N b_i\beta^{k_i}_{i,m}\cdot \psi_n (\lambda,x)\Big|\\
&=\Big|\bigast_{i=1}^N b_i\cdot \widehat{T}_{\overline{\alpha}^{\circ k_i}_i} \psi_n (\lambda, x)-\bigast_{i=1}^N b_i\beta^{k_i}_{i,m}\cdot \psi_n (\lambda,x)\Big|\\
&\leq \inf_{\eps>0}\omega_{\text{Id}}\Big(|\lambda|\sum_{i=1}^N  |b_i|\cdot\tn T_{\overline{\alpha}^{\circ k_i}_i}(u_n-u'_n)-\beta^{k_i}_{i,m}(u_{i,n}-u'_{i,n})\tn +\eps\Big)\\
&\leq \inf_{\eps>0}\omega_{\text{Id}}\Big(2|\lambda|\sum_{i=1}^N |b_i|\cdot\tn T_{\overline{\alpha}^{\circ k_i}_i}u_n-\beta^{k_i}_{i,m}u_{i,n}\tn +\eps\Big)\\
&\leq \inf_{\eps>0}\omega_{\text{Id}}\Big(2|\lambda|\Big(\sum_{i=1}^N |b_i|\cdot \left(\tn T_{\overline{\alpha}^{\circ k_i}_i}u_n-\beta^{k_i}_iu_n\tn+ |\beta^{k_i}_i-\beta^{k_i}_{i,m}|\cdot\tn u_n\tn\right)\Big) +\eps\Big)
\end{align*}\hfill

\noindent for all $(\lambda,x)\in \Q\times \Delta$. As the sequence $(u_n)_{n=1}^\infty$ is bounded (see Lemma \ref{N(u-v)}), taking the limit superiors over $n$ and $m$ in the inequality above yields the result.
\end{proof}

\section{Coarse $\ell_p$-types and coarse $c_0$-types. }\label{SectionCoarselp}

In this section, we will define a notion of  $\ell_p$-type and $c_0$-type and use  Proposition \ref{theoremspecialtype} in order to show that $\psi$ satisfies this property. Finally, we  will  show that $\overline{H}$ is isomorphic to $\ell_p$ for some $p\in[1,\infty )$.

\begin{defi} 
Let $p\in [1,\infty)$. We say that $\psi$ is a \emph{coarse $\ell_p$-type} if there exists $L>0$ such that, for all $(\lambda,x)\in \Q\times \Delta$, and all $\overline{\alpha}=(\alpha_i)_{i=1}^N\in\Q^{<\N}$, we have

\[\limsup_m\Big|\bigast_{i=1}^N\alpha_i\cdot\psi(\lambda,x)-t_m\cdot\psi(\lambda,x)\Big|\leq L.\]\hfill

\noindent for all $(t_m)_{m=1}^\infty\subseteq \mathbb{Q}$ converging to $\|\overline{\alpha}\|_p$. The type $\psi$ is called a \emph{coarse $c_0$-type} if, for all $(\lambda,x)\in \Q\times \Delta$, and all $(\alpha_i)_{i=1}^N\in\Q^{<\N}$, we have

\[\Big|\bigast_{i=1}^N\alpha_i\cdot\psi(\lambda,x)-\max_{1\leq i\leq N}|\alpha_i|\cdot\psi(\lambda,x)\Big|\leq L.\]\hfill

\end{defi}

\begin{prop}\label{ThmExistencelptypes}
The type $\psi$ is either a coarse $c_0$-type or a coarse $\ell_p$-type for some $p\in[1,\infty)$.
\end{prop}

\begin{proof}
Let $\overline{\alpha}_2=(1,1)$ and $\overline{\alpha}_3=(1,1,1)$, and notice that $\overline{\alpha}_2\circ\overline{\alpha}_3=\overline{\alpha}_3\circ\overline{\alpha}_2$. Let $\beta_2,\beta_3\in \R$ be given by Proposition \ref{theoremspecialtype} for $\overline{\alpha}_2=(1)$ and $\overline{\alpha}_3=(1,1)$, respectively. Let   $(\beta_{2,m})_{m=1}^\infty, (\beta_{3,m})_{m=1}^\infty\subseteq \mathbb{Q}$ be nonzero increasing  sequences converging to $\beta_2, \beta_3$ respectively. By our choice of $\beta_2$ and $\beta_3$, we have that

\[\limsup_m\Big|b\cdot\bigast_{i=1}^{j^k}\psi(\lambda,x)-b\beta_{j,m}^k\cdot\psi(\lambda,x)\Big|\leq \gamma\]\hfill
 
 \noindent for all $j\in \{2,3\}$, all $b\in \mathbb{Q}$, all $k\in\N$, and all $(\lambda,x)\in \Q\times \Delta$.

 Let $\ell,k\in\N$ be such that $3^{k}\leq 2^\ell< 3^{k+1}$.  As $(\xi_n)_{n=1}^\infty$ is $1$-unconditional,  we have

\[\TN  \sum_{i=1}^{3^k}\xi_i\TN  \leq \TN  \sum_{i=1}^{2^\ell}\xi_i\TN  \leq \TN  \sum_{i=1}^{3^{k+1}}\xi_i\TN  .\]\hfill

\noindent Let $a_\ell\in \mathbb{Q}$ be between $\frac{1}{2}\tn  \sum_{i=1}^{2^\ell}\xi_i\tn$ and $\tn  \sum_{i=1}^{2^\ell}\xi_i\tn$. Then, for any $\mu>0$,

\[\mu\leq \TN  \mu\cdot\frac{ \sum_{i=1}^{3^{k+1}}\xi_i}{a_\ell}\TN  .\]\hfill

As $\text{Id}\colon (X,\|\cdot\|)\to (X,d)$ is expanding, we can pick  $\mu\in \mathbb{Q}$ such that $\rho_\text{Id}(\mu/2)>2\omega_{\text{Id}}(1)+\gamma$ and $\eta\in \mathbb{Q}$ such that $\rho_\mathrm{Id}(\eta \|\xi_1\|/2)>2\omega_\mathrm{Id}(1)+\gamma$. Let $M\in \mathbb{N}$ be such that

\[\Big|\frac{\mu}{a_\ell}\cdot\bigast_{i=1}^{3^{k+1}}\psi(1,0)-\frac{\mu\beta_{3,M}^{k+1}}{a_\ell}\cdot\psi(1,0)\Big|\leq \gamma +\omega_\mathrm{Id}(1)\]\hfill

\noindent and let $N\geq M$ be such that

\[\Big|\frac{\eta}{\beta_{2,M}^\ell}\cdot\bigast_{i=1}^{2^\ell}\psi(1,0)-\frac{\eta\beta_{2,N}^\ell}{\beta_{2,M}^\ell}\cdot\psi(1,0)\Big|\leq \gamma +\omega_\mathrm{Id}(1).\]\hfill

\noindent  Then as  $(\mu/a_\ell)\cdot(\sum_{i=1}^{3^{k+1}}\xi_i) $ realizes $(\mu/a_l)\cdot \bigast_{i=1}^{3^{k+1}}\psi$,  by Lemma \ref{N(u-v)}(i),  we have that

\[2\omega_{\text{Id}}(1)+\gamma< \frac{\mu}{a_\ell}\cdot\bigast_{i=1}^{3^{k+1}}\psi(1,0)\leq \frac{\mu\beta^{k+1}_{3,M}}{a_\ell}\cdot\psi(1,0)+\gamma+\omega_\mathrm{Id}(1).\]\hfill

\noindent  Therefore, as $(\mu\beta_{3,M}^{k+1}/a_\ell)\cdot \xi_1$ realizes $(\mu\beta_{3,M}^{k+1}/a_\ell)\cdot\psi$,  by Lemma \ref{N(u-v)}(ii), we have

\begin{align}\label{Ineq1}
1\leq \frac{\beta_{3,M}^{k+1}\mu}{a_\ell}\cdot\tn  \xi_1\tn  .
\end{align}\hfill

Similarly, by Lemma \ref{N(u-v)}(i) and the fact that $(\eta \beta_{2,N}^\ell/\beta_{2,M}^\ell)\cdot \xi_1$ realizes $(\eta \beta_{2,N}^\ell/\beta_{2,M}^\ell)\cdot \psi$, we have 

 \[\frac{\eta}{\beta_{2,M}^\ell}\cdot \bigast_{i=1}^{2^\ell}\psi(1,0)\geq \frac{\eta\beta_{2,N}^\ell}{\beta_{2,M}^\ell}\cdot\psi(1,0)-\gamma-\omega_\mathrm{Id}(1)>\omega_{\text{Id}}(1).\]\hfill

\noindent  Hence, as $(\eta/\beta^\ell_{2,M})\cdot(\sum_{i=1}^{2^\ell}\xi_i)$ realizes $(\eta/\beta_{2,M}^\ell)\cdot \bigast_{i=1}^{2^\ell}\psi$, Lemma \ref{N(u-v)}(ii) gives us 

\begin{align}\label{Ineq2}
\frac{2\eta a_\ell}{\beta_{2,M}^\ell}\geq\TN  \frac{\eta\sum_{i=1}^{2^\ell}\xi_i}{\beta_{2,M}^\ell}\TN  \geq 1.
\end{align}\hfill

Combining Inequalities (\ref{Ineq1}) and (\ref{Ineq2}), we obtain 

\[\frac{\beta_3^k}{\beta_2^\ell}=\lim_M \frac{\beta_{3,M}^k}{\beta_{2,M}^\ell}\geq  \frac{1}{2\eta\mu \beta_3  \tn  \xi_1\tn  }.\]\hfill

\noindent The lower bound  for $\beta_3^k/\beta_2^\ell$ above does not depend on $k$ and $\ell$, as long as $2^\ell<3^{k+1}$. Similarly, we get a lower bound for $\beta_2 ^\ell/\beta_3^k$, which also does not depend on $k$ and $\ell$, as long as $3^k\leq 2^\ell$. We conclude that there exist $a,b>0$ such that for all $k$ and $\ell$, with $3^k\leq 2^{\ell}<3^{k+1}$, we have

\[a\leq \frac{\beta_3^k}{\beta_2^\ell}\leq b.\]\hfill

\noindent Therefore, there exists $L\geq 0$ such that $\beta_2=2^L$, and $\beta_3=3^L$. Also, as $\beta_2\leq 2$, we must have $L\in [0,1]$. The same argument works for arbitrary $n,m\in \mathbb{N}$ instead of $2$ and $3$. Hence, we have $\beta_n=n^L$, for all $n\in\N$, where $\beta_n$ is given by Proposition \ref{theoremspecialtype} for 

\[\overline{\alpha}=\underbrace{(1,\ldots ,1)}_{n}.\]\hfill

\noindent \textbf{Case 1:} Say $L\neq 0$. Then $\psi$ is a coarse $\ell_p$-type, for $p=1/L$.\\

Fix $\overline{\alpha}=(\alpha_i)_{i=1}^N\in \mathbb{Q}^N$ and a sequence $(t_m)_{m=1}^\infty\subseteq \mathbb{Q}$ converging to $\|\overline{\alpha}\|_p$. Let $\epsilon>0$ and, for each $1\leq j\leq N$, let $r_j\in \mathbb{Q}_+$ be such that $||\alpha_j|-r_j^{1/p}|<\eps$.
Find a common denominator $k\in \mathbb{N}$ so that for each $1\leq j\leq N$ there is $n_j\in \mathbb{N}_0$ such that $r_j=n_j/k$.
Let $s>0$ be a rational number such that $|s-(1/k)^{1/p}|<\eps$.
For each $1\leq j\leq N$, let $(\beta_{j,m})_{m=1}^\infty\subseteq \mathbb{Q}$ be a sequence converging to $n_j^{1/p}$and let $(\beta_{m})_{m=1}^\infty\subseteq\mathbb{Q}$ be a sequence converging to $(\sum_{j=1}^Nn_j)^{1/p}$.
By Lemma \ref{N(u-v)ext2} (and the symmetry of $\psi$),

\[\Big|\bigast_{j=1}^N\alpha_j\cdot \psi(\lambda,x)-\bigast_{j=1}^Ns\beta_{j,m}\cdot\psi(\lambda,x)\Big|\leq \omega_\mathrm{Id}\Big(|\lambda|\sum_{j=1}^N||\alpha_j|-s\beta_{j,m}|\tn \xi_j\tn+\eps\Big)\]\hfill

\noindent and

\[|s\beta_m\cdot \psi(\lambda,x)-t_m\cdot\psi(\lambda,x)|\leq \omega_\mathrm{Id}(|\lambda| |s\beta_m-t_m|\tn \xi_1\tn +\eps).\]\hfill

\noindent for all $(\lambda,x)\in \mathbb{Q}\times \Delta$.
By Proposition \ref{theoremspecialtype} and what was shown above with $L=1/p$, we have that

\[\limsup_m\Big|\bigast_{j=1}^N s\beta_{j,m}\cdot \psi(\lambda,x)-\bigast_{j=1}^Ns\cdot\bigast_{i=1}^{n_j} \psi(\lambda,x)\Big|\leq \gamma\]\hfill

\noindent and

\[\limsup_m\Big|s\cdot\bigast_{j=1}^N\bigast_{i=1}^{n_j}\psi(\lambda,x)-s\beta_{m}\cdot \psi(\lambda,x)\Big|\leq \gamma\]\hfill

\noindent for all $(\lambda,x)\in \mathbb{Q}\times\Delta$.

Combining the four inequalities above with the triangle inequality, taking a limit superior over $m$, and letting $\epsilon\to 0$, one obtains

\[\limsup_m\Big|\bigast_{j=1}^N\alpha_j\cdot \psi(\lambda,x)-t_m\cdot\psi(\lambda,x)\Big|\leq 4\gamma\]\hfill

\noindent for all $(\lambda,x)\in \mathbb{Q}\times \Delta$.
Therefore $\psi$ is a coarse $\ell_p$-type.
\\

\noindent \textbf{Case 2:} Say $L=0$. Then $\psi$ is a coarse $c_0$-type.\\

Fix $\overline{\alpha}=(\alpha_i)_{i=1}^N\in \mathbb{Q}^N$ such that $\alpha_1=1$ and $\alpha_j\leq 1$ for $2\leq j\leq N$ (the general case will follow by dilation).
Using Proposition \ref{theoremspecialtype}, find $\beta\geq 1$ and a nonzero increasing sequence $(\beta_m)_{m=1}^\infty\subseteq \mathbb{Q}$ converging to $\beta$ such that
 
\[\limsup_m |b\cdot\widehat{T}_{\overline{\alpha}}^k\psi(\lambda,x)-b\beta_m^k\cdot \psi(\lambda,x)|\leq \gamma\]\hfill

\noindent for all $b\in \mathbb{Q}$, $k\in \mathbb{N}$ and $(\lambda,x)\in \mathbb{Q}\times \Delta$.
We will show $\beta\leq 1$. Fix $k\in \mathbb{N}$ and note that $\widehat{T}_{\overline{\alpha}}^k\psi=\bigast_{i_k=1}^N\cdots\bigast_{i_1=1}^N(\prod_{\ell=1}^k\alpha_{i_\ell})\cdot \psi$ (using the definition of $\widehat{T}_{\overline{\alpha}}$ and the distributivity of dilation over convolution).
After combining like terms using the commutativity of convolution, by Proposition \ref{theoremspecialtype} and what was shown above with $L=0$, we have

\[ \Big|b\cdot\widehat{T}_{\overline{\alpha}}^k\psi(\lambda,x) - b\cdot\bigast_{\overline{n}\in F} (\prod_{j=1}^k\alpha_j^{n_j})\cdot\psi(\lambda,x)\Big|\leq \gamma\]\hfill

\noindent where $F=\{\overline{n}=(n_j)_{j=1}^k\in \mathbb{N}_0^k\ |\ \sum_{j=1}^k n_j=k\}$
for every $b\in \mathbb{Q}$ and $(\lambda,x)\in \mathbb{Q}\times \Delta$.
Now, take any $\mu\in \mathbb{Q}$ such that $\rho_\mathrm{Id}(\mu \tn \xi_1\tn/2)>2\omega_\mathrm{Id}(1)+2\gamma$.
Fix $M\in \mathbb{N}$, and let $N\geq M$ be such that $|\frac{\mu}{\beta_{M}^k}\widehat{T}_{\overline{\alpha}}^k\psi(1,0)-\frac{\mu \beta_{N}^k}{\beta_M^k}\cdot \psi(1,0)|\leq \gamma+\omega_\mathrm{Id}(1)$.
Then combining the two inequalities above yields

\[\Big|\frac{\mu\beta_N^k}{\beta_M^k}\cdot \psi(1,0)-\frac{\mu}{\beta_M^k}\cdot\bigast_{\overline{n}\in F} (\prod_{j=1}^k\alpha_j^{n_j})\cdot\psi(1,0)\Big|\leq 2\gamma +\omega_\mathrm{Id}(1).\]\hfill

\noindent As $(\mu \beta_N^k/\beta_M^k)\xi_1$ realizes $(\mu \beta_N^k/\beta_M^k)\cdot\psi$, we have, by Lemma \ref{N(u-v)}(i),
$\frac{\mu}{\beta_M^k}\cdot\bigast_{\overline{n}\in F} (\prod_{j=1}^k\alpha_j^{n_j})\cdot\psi(1,0)\geq \omega_\mathrm{Id}(1)$.
So, as $\frac{\mu}{\beta_M^k}\sum_{\overline{n}\in F} (\prod_{j=1}^k\alpha_j^{n_j})\cdot\xi_{I(\overline{n})}$ realizes $\frac{\mu}{\beta_M^k}\cdot\bigast_{\overline{n}\in F} (\prod_{j=1}^k\alpha_j^{n_j})\cdot\psi$ for any injective map $I\colon F\to\mathbb{N}$, we have, by Lemma \ref{N(u-v)}(ii),

\[1\leq \TN\frac{\mu}{\beta_M^k}\sum_{\overline{n}\in F} (\prod_{j=1}^k\alpha_j^{n_j})\cdot\xi_{I(\overline{n})}\TN\leq \frac{\mu \tn\xi_1\tn}{\beta_M^k}\prod_{\alpha_j<1}\frac{1}{1-\alpha_j}.\]\hfill

\noindent But this was for any $k,M\in \mathbb{N}$, and so we must have $\beta\leq 1$. That is, $\beta=1$. Therefore $\psi$ is a coarse $c_0$-type.
\end{proof}

We can now prove the following.

\begin{prop}\label{Thmlpspreading}
If $\psi$ is a coarse $\ell_p$-type, for some $p\in [1,\infty)$,  then $(\xi_n)_{n=1}^\infty$ is equivalent to the $\ell_p$-basis. If $\psi$ is a coarse $c_0$-type, then $(\xi_n)_{n=1}^\infty$ is equivalent to the $c_0$-basis. 
\end{prop}

\begin{proof}
Say $\psi\in\mathcal{T}$ is a coarse $\ell_p$-type for some $p\in [1,\infty)$ (the $c_0$ case will be analogous). Say $L>0$ is such that, for any $(\alpha_j)_{j=1}^N\in \Q^{<\mathbb{N}}$, any $(t_m)_{m=1}^\infty\subseteq \mathbb{Q}$ converging to $\|\overline{\alpha}\|_p$, and all $(\lambda, x)\in \Q\times \Delta$, we have

\begin{align}\label{Iqfinal2}
\limsup_m\Big|\bigast_{j=1}^N\alpha_j\cdot\psi(\lambda,x)-t_m\cdot \psi(\lambda, x)\Big|\leq L.
\end{align}\hfill

Let $(e_n)_{n=1}^\infty$ be the standard basis of $\ell_p$, and let $Y=\text{span}\{e_n\mid n\in\N\}$. Let us show that the map $T\colon Y\to \text{span}\{\xi_n\mid n\in\N\}$ defined by sending $e_n$ to $\xi_n/\tn\xi_1\tn$ for each $n\in \mathbb{N}$ and extending linearly is an isomorphism.  Hence, $T$ extends to an isomorphism between $\ell_p$ and $\overline{\text{span}}\{\xi_n\mid n\in\N\}$, and we are done.

We first show that $T$ is bounded. Fix a small $\eps>0$ and let $b\in\Q$ be such that $1/\tn\xi_1\tn<b<1/\tn\xi_1\tn+\eps$. For each $\overline{\alpha}=(\alpha_i)_{i=1}^N\in\Q^{<\N}$, let $t_{\overline{\alpha}}\in \mathbb{Q}$ be such that $|t_{\overline{\alpha}} -\|\overline{\alpha}\|_p|<\eps$ and $|\bigast_{j=1}^N\alpha_j\cdot\psi(b,0)-t_{\overline{\alpha}}\cdot \psi(b, 0)|\leq L+\eps$. By Lemma \ref{N(u-v)} and Inequality \ref{Iqfinal2}, we have that

\begin{align*}
\rho_{\text{Id}}\Big(\TN  \sum_{i=1}^N \alpha_i \frac{\xi_j}{\tn\xi_1\tn}\TN -\eps\Big) &\leq \rho_{\text{Id}}\Big(\TN  \sum_{i=1}^N \alpha_i b\xi_j\TN -\eps\Big)\\
&\leq \bigast_{i=1}^N \alpha_i\cdot \psi(b,0)\\
&\leq t_{\overline{\alpha}}\cdot\psi(b,0)+L+\eps\\
&\leq \omega_\text{Id}\Big(b\tn  \xi_1\tn   t_{\overline{\alpha}}+\eps\Big)+L+\eps\\
&\leq \omega_\text{Id}\Big(\|\overline{\alpha}\|_p+2\eps +\eps\|\xi_1\| \|\overline{\alpha}\|_p+\eps^2\|\xi_1\|\Big)+L+\eps,
\end{align*}\hfill

\noindent for all $\overline{\alpha}=(\alpha_i)_{i=1}^N\in \Q^{<\N}$. Hence, as $\text{Id}\colon(X,\|\cdot\|)\to (X,d)$ is expanding,  there exists $K>0$ such that $\|\overline{\alpha}\|_p\leq 1$ implies $\tn  \sum_{i=1}^N\alpha_i \frac{\xi_i}{\tn\xi_1\tn}\tn  \leq K$. Therefore $T$ is bounded.

Clearly, $T$ is a bijection. Let us show that $T^{-1}$ is bounded. By Lemma \ref{N(u-v)} and Inequality \ref{Iqfinal2}, we have that

\begin{align*}
\rho_\text{Id}\Big(\|\overline{\alpha}\|_p-2\eps\Big)-L-\eps &\leq
\rho_\text{Id}\Big(bt_{\overline{\alpha}}\tn\xi_1\tn-\eps\Big)-L-\eps\\
&\leq t_{\overline{\alpha}}\cdot \psi(b,0)-L-\eps\\
&\leq  \bigast_{i=1}^N \alpha_i\cdot\psi(b,0)\\
&\leq \omega_\text{Id}\Big(b\TN  \sum_{i=1}^N \alpha_i\xi_i\TN+\eps   \Big)\\
&\leq \omega_\text{Id}\Big( \TN  \sum_{i=1}^N \alpha_i\frac{\xi_i}{\tn\xi_1\tn}\TN+\eps \tn\xi_1\tn \TN  \sum_{i=1}^N \alpha_i\frac{\xi_i}{\tn\xi_1\tn}\TN   \Big).
\end{align*}\hfill

\noindent for all $\overline{\alpha}=(\alpha_i)_{i=1}^N\in \Q^{<\N}$. Hence, as $\text{Id}\colon(X,\|\cdot\|)\to (X,d)$ is expanding, there exists some $R>0$ such that $\tn\sum_{i=1}^N \alpha_i\frac{\xi_i}{\tn\xi_1\tn}\tn\leq 1$ implies $\|\overline{\alpha}\|_p<R$. So $T^{-1}$ is bounded.
\\
\\
\end{proof}

We now have everything needed to prove the main theorem.

\begin{maintheorem}
If a Banach space $X$ coarsely embeds into a superstable Banach space, then $X$ has a basic sequence with an associated spreading model isomorphic to $\ell_p$, for some $p\in[1,\infty)$.
\end{maintheorem}

\begin{proof}
By Corollary \ref{invmetric2}, if $X$ coarsely embeds into a superstable space $Y$, there exists a translation-invariant stable pseudometric $d$ on $X$ which is coarsely equivalent to the metric induced by the norm of $X$. Hence, we can define the type space $\mathcal{T}$ as in Section \ref{SectionType}.  By Proposition \ref{minimal}, there exists a minimal closed admissible conic class $\mathcal{C}$. Let $\phi\in\mathcal{C}$ be given by Lemma \ref{commonpoint}. Without loss of generality,  $\psi=\phi*(-1)\cdot\phi$ is admissible. By Proposition \ref{ThmExistencelptypes}, $\psi$ is either a $c_0$-type or an $\ell_p$-type, for some $p\in [1,\infty)$.  Hence, by Proposition \ref{Thmlpspreading}, $X$ has a spreading model isomorphic to either $\ell_p$ for some $p\in[1,\infty)$ or to $c_0$. 

Assume that $X$ has a spreading model isomorphic to $c_0$.  In particular, $c_0$ is finitely represented in $X$. Hence, $c_0$ (isometrically) isomorphically embeds into an ultrapower of $X$. As ultrapowers of $X$ coarsely embed into ultrapowers of $Y$, this gives us that $c_0$ coarsely embeds into an ultrapower of $Y$, which is a stable space.  By \cite[Theorem 2.1]{Kalton2007}, stable spaces coarsely embed into reflexive spaces. Therefore, $c_0$ coarsely embeds into a reflexive space. By \cite[Theorem 3.6]{Kalton2007}, this cannot happen, so we have a contradiction. Therefore, $X$ contains a spreading model isomorphic to $\ell_p$, for some $p\in [1,\infty)$.

Let $(x_n)_{n=1}^\infty$ be a bounded sequence in $X$ without Cauchy subsequences whose spreading model is isomorphic to $\ell_p$. Let us observe now that $(x_n)_{n=1}^\infty$ can be assumed to be a basic sequence. By Rosenthal's $\ell_1$-Theorem, either $(x_n)_{n=1}^\infty$ has a subsequence which is equivalent to the basis of $\ell_1$, or it has a weakly Cauchy subsequence. Assume that $(x_n)_{n=1}^\infty$ is weakly Cauchy. Then $(y_n)_{n=1}^\infty$ is weakly null and it has a spreading model isomorphic to $\ell_p$, where $y_n=x_n-x_{n+1}$, for all $n\in\N$. Hence, by taking a subsequence, we can assume that $(y_n)_{n=1}^\infty$ is basic. 
\end{proof}

\begin{rem}\label{constant}
By the last inequality of Case 1 in Proposition \ref{ThmExistencelptypes}, and by following the proof of Proposition \ref{Thmlpspreading}, we find an upper bound of 

\[\Big(\inf_{\eps>0}\sup \rho_\mathrm{Id}^{-1}([0,\omega_\mathrm{Id}(1)+5\gamma +\eps])\Big)^2\]\hfill

\noindent for the Banach-Mazur distance between $\ell_p$ and the spreading model associated to $(y_n)_{n=1}^\infty$.  This is the best the authors have been able to obtain using the methods of Raynaud.
\end{rem}

\begin{maincorollary}
There are separable reflexive Banach spaces which do not coarsely embed into any superstable Banach space. 
\end{maincorollary}

\begin{proof}
This follows from the  fact that the original Tsirelson space (see \cite{Tsirelson1974}) does not have a spreading model isomorphic to $\ell_p$ for any $p\in[1,\infty)$. 
\end{proof}

\begin{remark}
Another example of a reflexive Banach space that does not coarsely embed into any superstable space is the space constructed by Odell and Schlumprecht in  \cite{OdellSchlumprecht1995}. Indeed, this follows from Theorem \ref{main} and the fact that  every spreading model of their space contains neither a subspace isomorphic to $c_0$ nor to $\ell_p$ (see \cite[Theorem 1.4]{OdellSchlumprecht1995}).
\end{remark}

As mentioned in the introduction, our work is not enough to solve Problem \ref{ProbKalton}. The following is a  natural approach to give a negative answer to Problem \ref{ProbKalton}, given Theorem \ref{main}.

\begin{problem}
Let $T$ be the Tsirelson space defined in \cite{FigielJohnson1974}, i.e., the dual of the original space of Tsirelson defined in \cite{Tsirelson1974}. Does $T$ or $T^{(p)}$ (i.e., the $p$-convexification of $T$) for some $p\in [1,\infty)$ coarsely embed into a superstable Banach space?
\end{problem}

\textbf{Acknowledgments:} The first author would like to thank his adviser, Christian Rosendal, for all the help during the preparation of this paper.  He would also like to thank the organizers of the Workshop in Analysis and Probability:  D. Kerr, W. B. Johnson, and G. Pisier.  Part of this research occurred during the workshop held in July 2016.  The second author was partly supported by NSF grants DMS-1464713 and DMS-1565826.  He would like to thank his advisors Thomas Schlumprecht and Florent Baudier for their endless support.

\end{document}